\documentclass[12 pt]{article}
\usepackage{amsmath, amsthm, amssymb,enumerate, comment}
\usepackage{tikz}
\usetikzlibrary{positioning,chains,fit,shapes,calc}
\usepackage{fullpage}
\usepackage{enumitem}
\usepackage{hyperref}
\usepackage{xcolor}
\usepackage[font=it]{caption}
\usepackage{algorithm,algorithmic}
\usepackage{caption}

\newlength\myindent
\title{Maximising the number of induced cycles in a graph\footnote{Work partially done during the 2015 Barbados Workshop on Structural Graph Theory. The attendance of the first author was possible due to the generous support of Merton College Oxford.}}
\author{Natasha Morrison\thanks{Mathematical Institute, University of Oxford, Woodstock Road, Oxford, OX2 6GG, United Kingdom. \newline  E-mail: \texttt{\{morrison,scott\}@maths.ox.ac.uk}.} \and Alex Scott\footnotemark[2]} 
\newtheorem{thm}{Theorem}[section]
\newtheorem{lem}[thm]{Lemma}
\newtheorem{prop}[thm]{Proposition}
\newtheorem{conj}[thm]{Conjecture}
\newtheorem{cor}[thm]{Corollary}

\theoremstyle{definition}
\newtheorem{defn}[thm]{Definition}

\newtheorem{ques}[thm]{Question}
\newtheorem{claim}[thm]{Claim}

\numberwithin{equation}{section}

\newtheoremstyle{case}{}{}{\normalfont}{}{\itshape}{\normalfont:}{ }{}

\theoremstyle{case}

\theoremstyle{case}
\newtheorem{subcase}{\bf{Subcase}}

\theoremstyle{case}

\theoremstyle{case}
\newtheorem{case2}{Case}

\def\comment#1{}

\numberwithin{equation}{section}

\begin{document}

\maketitle

\begin{abstract}
We determine the maximum number of induced cycles that can be contained in 
 a graph on $n\ge n_0$ vertices, and show that there is a unique graph that 
 achieves this maximum. This answers a question of Chv\'{a}tal and Tuza from the 1980s. 
 We also determine the maximum number of odd or even induced cycles that can be 
 contained in a graph on $n\ge n_0$ vertices and characterise the extremal 
 graphs.  This resolves a conjecture of Chv\'{a}tal and Tuza 
 from 1988.

\end{abstract}

\section{Introduction}
What is the maximum number of induced cycles in a graph on $n$ vertices?  For cycles of {\em fixed} length, this problem has been extensively studied.
Indeed, for any fixed graph $H$, let the {\em induced density} of $H$ in a graph $G$ be the number of induced 
copies of $H$ in $G$ divided by $\binom{|G|}{|H|}$; let $I(H;n)$ be the maximum induced density of $H$ over all graphs $G$ on $n$ vertices;
and let the {\em inducibility} of $H$ be the limit $\lim_{n\to\infty}I(H;n)$. 
In 1975, Pippinger and Golumbic \cite{pip} made the following 
conjecture.

\begin{conj}\cite{pip}
For $k\ge5$, the inducibility of the cycle $C_k$ is $k!/(k^k - k)$.
\end{conj}

Balogh, Hu, Lidick\'{y} and Pfender \cite{bal} recently proved this conjecture 
in the case $k=5$ via a flag algebra method, and showed that the maximum density 
was achieved by a unique graph. Apart from this case, the problem remains open (though see
\cite{behj,bnt,bs,el,exoo,hhn,h} 
for results on inducibility of other graphs).

In this paper, we consider the total number of induced cycles, without restriction on length.  This problem was raised in the 1980s by Chv\'{a}tal and Tuza (see \cite{tuz} and \cite{tuz3}), who asked for the 
 maximum possible number of induced cycles in a graph with $n$ vertices. 
 The problem was investigated independently in unpublished work of Robson, who 
 showed in the 1980s that a graph on $n$ vertices has at most 
 $3^{(1+o(1))n/3}$ induced cycles (\cite{bm,mr}). Tuza also raised a second, closely related problem 
 on induced cycles. In 1988 he 
  conjectured with Chv\'{a}tal (see \cite{tuz2}, \cite{tuz} and \cite{tuz3}) that the maximum possible number of odd induced 
  cycles in a graph on $n$ vertices is $3^{n/3}$.
  
  In this paper we resolve both problems, proving exact bounds for all 
   sufficiently large $n$, and determining the extremal graphs. Our methods 
   work for a number of problems of this type: thus we will determine, for 
   sufficiently large $n$, the graphs with $n$ vertices that maximize the number 
   of induced cycles (Theorem \ref{main}); we will also determine the graphs with the maximum 
   number of even induced cycles, the maximum number of odd induced cycles (Theorem \ref{oddcycle}), and in Theorem \ref{oddholes}, the maximum 
   number of odd holes (i.e. induced odd cycles of length at least 5).
  
In order to state our results, it is helpful to have a couple of definitions.
As usual, for $G$ a graph define the \emph{neighbourhood} of $x$ to be 
$N_{G}(x):= \{y \in V(G): xy \in E(G)\}$. A graph $B$ is called a \emph{cyclic braid} if there exists $k \ge 3$ and a partition 
$B_1,\ldots, B_k$ of $V(B)$ such that for every $1 \le i \le k$ 
and every $x \in B_i$, we have $B_{i-1}\cup B_{i+1} \subseteq N_{B}(x) 
\subseteq B_{i-1}\cup B_{i} \cup B_{i+1}$ where indices are taken modulo $k$. For such a partition, the notation $B = (B_1,\ldots,B_k)$ is used. The sets $B_i$ are called \emph{clusters} of $B$; the \emph{length} of 
the cyclic braid is the number of clusters. If a cyclic braid contains no edges 
within its clusters, it is called an \emph{empty cyclic braid}. Observe that an empty cyclic braid is $k$-partite and, when $k>3$, is triangle free. If a cyclic braid contains 
every possible edge within each cluster, then it is called \emph{full}. A pair of 
clusters $B_1$ and $B_2$ are \emph{adjacent} in $G$ if $v_1v_2 \in E(G)$ for all 
$v_1 \in B_1$ and $v_2 \in B_2$. A triple of clusters $B_1, B_2,B_3$ are 
\emph{consecutive} if $B_1$ is adjacent to $B_2$ and $B_2$ is adjacent to $B_3$.

As it turns out, the structure of the extremal graph depends on the value of 
$n$ modulo 3. For $n \ge 8$ define an $n$-vertex graph $H_n$ 
separately for each value of $n$ modulo 3. Let $k \ge 3$. Define $H_{3k}$ to 
be the empty cyclic braid of length $k$ where every cluster has size 3. Define $H_{3k +1}$ to be the empty cyclic braid containing $k-1$ clusters of 
size 3 and one of size 4. Finally, define $H_{3k-1}$ to be the empty cyclic 
braid containing $k-1$ clusters of size 3 and one of size 2. 

Let $m(n)$ be the maximum number of induced cycles that can be contained in a graph on $n$ vertices. The main result of our paper is the following.

\begin{thm}\label{main}
There exists $n_0$ such that, for all $n \ge n_0$, $H_n$ is the unique graph on $n$ vertices containing $m(n)$ induced cycles.
\end{thm}

More precisely, 

\begin{cor}\label{maincor}
There exists $n_0$ such that, for all $n \ge n_0$: 
$$m(n) = \left\{ \begin{array}{c l}      
    3^{n/3} + 12n & \text{ for } n \equiv 0 \text{ modulo 3};\\
    4 \cdot 3^{(n-4)/3} +  12n + 51 & \text{ for } n \equiv 1 \text{ modulo 3};\\
    2 \cdot 3^{(n-2)/3} + 12n - 36 & \text{ for } n \equiv 2 \text{ modulo 3}.\\
\end{array}\right.$$ 
\end{cor}
Corollary \ref{maincor} implies that $m(n) = \Theta(3^{n/3})$.  

A \emph{hole} is an induced cycle of length at least 4. For $n \ge 10$, the graph $H_n$ is triangle free and each induced cycle is a hole, so Theorem \ref{main} implies the following.

\begin{cor}
There exists $n_0$ such that, for all $n \ge n_0$, $H_n$ is the unique graph on $n$ vertices with the maximum number of holes.
\end{cor}

Using similar arguments to those in the proof of Theorem \ref{main} we also prove a stability-type result. 
\begin{thm}\label{stab}
Fix $0 < \alpha < 1$. There exists a constant $C= C(\alpha)$ and $n_0 = n_0(\alpha)$ such that for any $n \ge n_0$, if a graph $F$ on $n$ vertices contains at least $\alpha \cdot m(n)$ induced cycles, then by adding or deleting edges incident to at most $C(\alpha)$ vertices of $F$, the graph $F$ can be transformed into a cyclic braid with the same cluster sizes as $H_n$.
\end{thm}

The arguments used to prove Theorem \ref{main} can be adapted to give results about other sets of induced cycles,
for instance induced cycles of given parity.
We say that a path or cycle is \emph{odd} if it contains an odd number of vertices (\emph{even} if 
it contains an even number). 

Let $m_o(n)$ be the maximum number of induced odd cycles that can 
be contained in a graph on $n$ vertices.  The value of $m_o(n)$ and the structure of the extremal graphs
depend on the value of $n$ modulo 6.

Define $G_n$ to be the full cyclic braid on $n$ vertices whose clusters all have size 3 except for:
\begin{itemize}
\item three consecutive clusters of size 2, when $n \equiv 0$ modulo 6;
\item two adjacent clusters of size 2, when $n \equiv 1$ modulo 6; 
\item one cluster of size 2, when $n \equiv 2$ modulo 6;
\item one cluster of size 4, when $n \equiv 4$ modulo 6;
\item two adjacent clusters of size 4, when $n \equiv 5$ modulo 6.
\end{itemize} 

We will prove the following.
 
\begin{thm}\label{oddcycle}
There exists $n_0$ such that, for all $n \ge n_0$, $G_n$ is the unique $n$-vertex graph containing $m_o(n)$ induced odd cycles.
\end{thm}

 This resolves the conjecture of Chv\'{a}tal and  Tuza (see \cite{tuz2}, \cite{tuz} and \cite{tuz3}), showing that $m_o(n)$ is within an $O(n)$ additive term of the conjectured $3^{n/3}$ when $n \equiv 3$ modulo 6 and within a constant factor when $n \not\equiv 3$ modulo 6. 
 
 If we consider odd holes (induced odd cycles of length at least 5), we get the same bound but a larger family of extremal graphs. 
 Let $m_o^h(n)$ be the 
 maximum number of odd holes that can be contained in a graph on $n$ vertices. Define 
$\mathcal{G}_n$ to be the family of cyclic braids on $n$ vertices whose cluster 
sizes are the same as the cluster sizes in $G_n$, but with no restrictions 
on which clusters are adjacent or consecutive, or on which edges are present inside the clusters; in addition, when $n \equiv 
5$ modulo 6, we also include the cyclic braids whose clusters all 
have size 3 except for four clusters of size 2. 

A modification of the proof of Theorem \ref{oddcycle} gives the following.

\begin{thm}\label{oddholes}
 There exists $n_0$ such that, for all $n \ge n_0$, the family of $n$-vertex graphs that contain $m_o^h(n)$ odd holes is $\mathcal{G}_n$.
\end{thm}

Let $m_e(n)$ be the maximum number of induced even cycles that can be 
contained in a graph on $n$ vertices, and define $E_n$ to be the empty cyclic 
braid on $n$ vertices whose clusters all have size 3 except for:
\begin{itemize}
\item one cluster of size 4, when $n \equiv 1$ modulo 6;
\item two adjacent clusters of size 4, when $n \equiv 2$ modulo 6;
\item three consecutive clusters of size 2, when $n \equiv 3$ modulo 6;
\item two adjacent clusters of size 2, when $n \equiv 4$ modulo 6; and
\item one cluster of size 2, when $n \equiv 5$ modulo 6.
\end{itemize} 

The proof of Theorem \ref{oddcycle}, adapted to consider even rather than odd induced cycles, gives the following. 

\begin{thm}\label{even}
There exists $n_0$ such that, for all $n \ge n_0$, $E_n$ is the unique $n$-vertex graph containing $m_e(n)$ induced even cycles. 
\end{thm}
As in the case of $m_o(n)$, we have $m_e(n) = \Theta(3^{n/3})$.
%\begin{cor}
% There exists $n_0$ such that, for all $n \ge n_0$, we have that
% $$m_e(n) = \left\{ \begin{array}{c l}      
 %    2^3\cdot 3^{n/3} + 12n & \text{ for } n \equiv 0 \text{ modulo 6}\\
  %   4 \cdot 3^{(n-4)/3} +  12n + 51  & \text{ for } n \equiv 1 \text{ modulo 6}\\
   %  4^2 \cdot 3^{(n-8)/3} + 12n + 129  & \text{ for } n \equiv 2 \text{ modulo 6}\\
    % 2^3 \cdot 3^{(n-6)/3} + 12n - 75 & \text{ for } n \equiv 3 \text{ modulo 6}\\
    % 4 \cdot 3^{(n-4)/3} +  12n - 56  & \text{ for } n \equiv 4 \text{ modulo 6}\\
     %2 \cdot 3^{(n-2)/3} + 12n - 36 & \text{ for } n \equiv 5 \text{ modulo 6}\\
 %\end{array}\right.$$ 
%\end{cor}

The paper is structured as follows. In Section \ref{section:path} we prove a preliminary result (Theorem~\ref{thm:path}) determining the structure of the $n$-vertex graphs that maximise the number of induced paths between a particular pair of vertices. During the proof of this result we will introduce several of the key ideas needed later. The proof of the main theorem (Theorem~\ref{main}, maximising the number of induced cycles) is given in Section \ref{section:cycle}. The proof uses Theorem \ref{thm:path} from Section \ref{section:path}, but otherwise is entirely contained in Section~\ref{section:cycle}. A number of lemmas proved in Section \ref{section:cycle} are proved in more generality than is strictly needed. This is because the more general versions will be used in Section \ref{sectIon: stab}, where we prove Theorem \ref{stab}. The proofs of Theorem \ref{oddcycle}, Theorem \ref{oddholes} and Theorem \ref{even} are given in Section \ref{section:odd}. However, all the key ideas for the proofs are the same as those in the proof of Theorem \ref{main}. Finally, in Section \ref{section:end} we conclude by discussing some open questions. 

\section{Induced paths between a pair of vertices}\label{section:path}

Let $G$ be a finite graph and let $x$ and $y$ be distinct vertices in $V(G)$. Define $p_2(G;x,y)$ to be the number of induced paths in $G$ beginning at $x$ and ending at $y$. Also define:
$$p_2(G) := \max \{p_2(G;x,y): x,y \in V(G)\},$$ and $$p_2(n):= \max\{p_2(G): |V(G)| = n\}.$$
We use the notation $p_2(\cdot)$, as it indicates that we are counting the maximum number of induced paths between \emph{two} fixed vertices. 

Our first goal in this section is to determine the structure of the $n$-vertex graphs that contain $p_2(n)$ induced paths between some pair of vertices. We show that these extremal graphs have a particular structure that depends on the value of $n$ modulo 3. We then prove analogous results for odd and even length paths. The strategy and notation used in this section for paths are later developed for induced cycles in Section \ref{section:cycle}.

Let $F$ be a graph and let $B_1, \ldots, B_k$ be disjoint subsets of $V(F)$. A subgraph $B \subseteq F$ is a \emph{braid} in $F$ if there exists $k \ge2$ and a partition $B_1,\ldots,B_k$ of $V(B)$ such that for each $2\le i \le k-1$ and for every $x \in B_i$ we have $B_{i-1} \cup B_{i+1} \subseteq N_{F}(x) \subseteq B_{i-1} \cup B_i \cup B_{i+1}$. If $V(F) = \bigcup_{i=1}^k B_i $, we say that $F$ is a \emph{braid}. For such a partition, the notation $B = (B_1,\ldots,B_k)$ is used. The sets $B_i$ are called \emph{clusters}. If $i \in \{1,k\}$ we say $B_i$ is an \emph{end cluster}; otherwise we say $B_i$ is a \emph{central cluster}.  The \emph{length} of a braid is the number of clusters it contains.

Let $n \ge 4$. Define $\mathcal{F}_n$ to be the set of all braids $B$ with the following properties: 
\begin{itemize}
\item $|V(B)| = n$.
\item $B$ has end clusters of size one.
\item  If $n \equiv 0$ modulo 3, then either one central cluster has size 4 and the rest have size 3, or two have size two and the rest have size 3.
\item If $n \equiv 1$ modulo 3, then one central cluster has size 2 and the rest have size 3.
\item If $n \equiv 2$ modulo 3, all clusters have size 3. 
\end{itemize}
Observe that there are no conditions on whether clusters in the braid contain edges. See Figure \ref{pathex} for an example of a braid in $\mathcal{F}_{10}$.
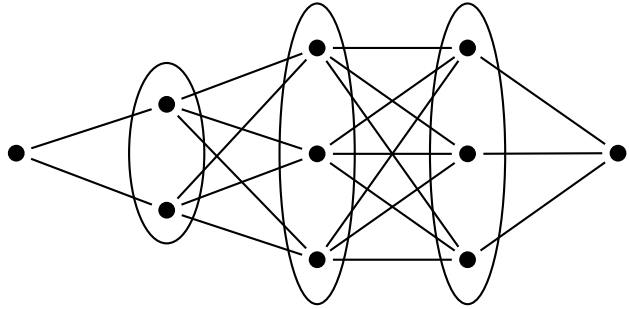
\begin{figure}[h]
\centering 

\begin{tikzpicture}[thick,
  every node/.style={draw,circle}, 
  fsnode/.style={fill=black,scale=0.5},
  ssnode/.style={fill=black,scale=0.5},
  wsnode/.style={fill=black,scale=0.5},
  rsnode/.style={fill=black,scale=0.5},
  every fit/.style={ellipse,draw,inner sep=45pt,text width=0.1cm},
  ->,shorten >= 3pt,shorten <= 3pt
]

\node[draw=none,fill=none] at (0,-1.4) (y) {};

% the vertices of Y
\begin{scope}[yshift = -.75cm, start chain=going below,node distance=12mm]
\foreach \i in {1,2}
  \node[fsnode,on chain] (f\i){};
\end{scope}

% the vertices of Z
\begin{scope}[xshift=2cm,yshift=-0cm,start chain=going below,node distance=12mm]
\foreach \i in {1,2,...,3}
  \node[ssnode,on chain] (s\i)  {};
\end{scope}

% the vertices of W
\begin{scope}[xshift=4cm,yshift=-0cm,start chain=going below,node distance=12mm]
\foreach \i in {1,2,...,3}
  \node[wsnode,on chain] (w\i)  {};
\end{scope}

\begin{scope}[xshift=6cm,yshift=-14mm,start chain=going below,node distance=12mm]
\foreach \i in {1}
  \node[rsnode,on chain] (r\i)  {};
\end{scope}

\begin{scope}[xshift = -2cm, yshift=-14mm,start chain=going below,node distance=12mm]
\foreach \i in {1}
  \node[rsnode,on chain] (p\i)  {};
\end{scope}

\draw (p1) edge[-] (f1)(p1) edge[-] (f2) (f1) edge[-] (s1) (f2) edge[-] (s2)  (f1) edge[-] (s2) (f1) edge[-] (s3) (f2) edge[-] (s1) (f2) edge[-] (s3) 
(w1) edge[-] (s1) (w2) edge[-] (s2) (w3) edge[-] (s3) (w1) edge[-] (s2) (w1) edge[-] (s3) (w2) edge[-] (s1) (w2) edge[-] (s3) (w3) edge[-] (s1) (w3) edge[-] (s2)
(w1) edge[-] (r1) (w2) edge[-] (r1) (w3) edge[-] (r1)
;

\draw (y) ellipse (0.5cm and 1.2cm);
\draw (s2) ellipse (0.5cm and 2cm);
\draw (w2) ellipse (0.5cm and 2cm);

\end{tikzpicture}
\caption{An example of a braid in $\mathcal{F}_{10}$.}
\label{pathex}
\end{figure}

Let the end clusters be $\{x\}$ and $\{y\}$. Every graph in $\mathcal{F}_n$ contains the same number of induced paths between $x$ and $y$ and so we define 

$$f_2(n) = \left\{ \begin{array}{c l}      
    3^{(n-2)/3}  & \text{ for } n \equiv 2 \text{ modulo 3}\\
    4 \cdot 3^{(n-6)/3}  & \text{ for } n \equiv 0 \text{ modulo 3} \\
    2 \cdot 3^{(n-4)/3} & \text{ for } n \equiv 1 \text{ modulo 3}\\
\end{array}\right.$$
and observe that $p_2(F) = p_2(F;x,y) = f_2(n)$ for all $F \in \mathcal{F}_n$. 

In this section we prove the following.

\begin{thm}\label{thm:path}
Let $G$ be a finite graph on $n \ge 4$ vertices and let $x$ and $y$ be distinct vertices of $G$. Suppose that $G$ contains $p_2(n)$ induced paths between $x$ and $y$.
Then $G$ is isomorphic to a graph in $\mathcal{F}_n$ with end clusters $\{x\}$ and $\{y\}$.
\end{thm}

In particular, this gives the following.
\begin{cor}\label{pathcor}
For all $n \ge 4$, we have $p_2(n) = f_2(n)$. 
\end{cor}

Before proving Theorem \ref{thm:path}, we introduce some preliminary notation and definitions.

\begin{defn}
\label{treexy}
Define $N[v] := N(v) \cup \{v\}$. Also, for a set $X \subseteq V(G)$, let $N(X) := \bigcup_{x \in X}N(x)$, and $N[X] := \bigcup_{x \in X}N[x]$. Note that $X \subseteq N[X]$, and $X \cap N(X)$ may or may not be empty. For a subgraph $H \subseteq G$, define $N(H) := N(V(H))$ and $N[H]:= N[V(H)]$.
\end{defn}

In order to prove Theorem \ref{thm:path} (counting induced paths), the following definition is used. A similar definition is given in Definition \ref{etree} to prove Theorem \ref{main} (counting induced cycles).  

\begin{defn}
Let $G$ be a finite graph and fix $x,y \in V(G)$, with $y \notin N[x]$. The \emph{x-y-path tree} of $G$ is a tree $T = T(x,y)$ together with a function $t:V(T) \rightarrow V(G)$ defined as follows.
\begin{itemize}
 \item $T$ is a tree with vertex set $V(T)$ disjoint from $V(G)$. 
 \item The vertices of $T$ correspond to induced paths $P := x,x_1\ldots,x_j$ in $G$ such that $y \notin N_G(\{x_1,\ldots,x_{j-2}\})$ and $y \in N(x_{j-1})$ only if $x_j = y$. For every such $P$ define a vertex $w_P$ in $T$, and set $t(w_P) = x_j$. We say that $P$ is the \emph{G-path} of $w_P$.  These vertices are the only vertices in $T$. Define the root of the tree to be $v_0:= w_{x}$.
 \item Given a vertex $w \in V(T)$ with $G$-path $x,x_1\ldots,x_j$ we define $C(w)$, the \emph{children} of $w$, to be the set of vertices in $T$ whose $G$-path is $x,x_1,\ldots, x_j,z$, for some $z \in G$. Define $N_T(v_0) = C(v_0)$. For $w \in V(T)\backslash \{v_0\}$  define $N_T(w) := C(w) \cup \{u\}$ where $u$ is the unique vertex in $T$ with $G$-path $x,x_1\ldots,x_{j-1}$. (So two vertices are adjacent in $T$ precisely when one of their $G$-paths extends the other by one vertex.) 
\end{itemize} 
We write $t(S):= \{t(s):x \in S\}$ for any subset $S \subseteq V(T)$ and $t(H):= G[\{t(x): x \in V(H)\}]$ for any subgraph $H \subseteq T$. 
Given a set $S \subseteq V(T)$, we say that $S$ \emph{sees} a vertex $w \in V(G)$ (or $w$ is \emph{seen} by $S$) if $w \in N_G[t(S)]$. An empty set does not see any vertex. If $w \not\in N_G[t(S)]$, we say that $w$ is \emph{unseen} by $S$ (or that $S$ does not \emph{see} $w$).
\end{defn}
See Figure \ref{treex} for an example of an \emph{x-y-path tree}. We get the following proposition as an easy consequence of Definition \ref{treexy}. 

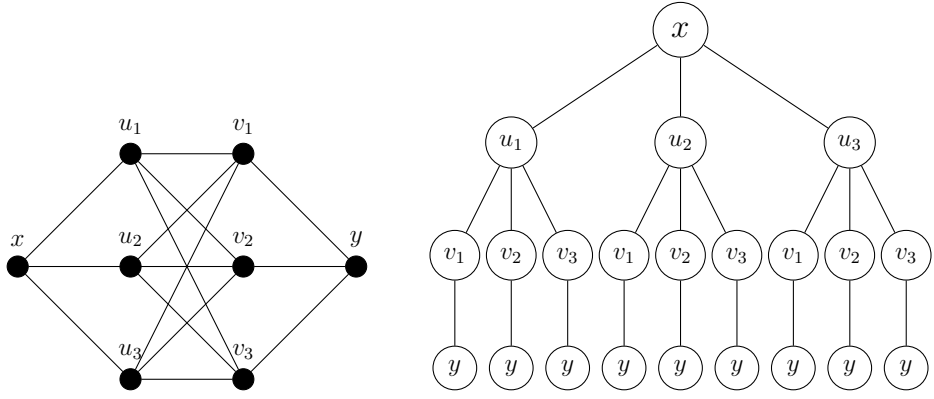
\begin{figure}[htbp]
\centering
\begin{tikzpicture}[node distance=2.5cm,main node/.style={minimum size = .7cm,circle,fill=white!20,draw}, b node/.style={circle,fill=white!20,draw}, scale=1.5, every node/.style={scale=0.8}]

\node[main node] (x)at (0,1) [fill=black, scale=0.5, label={above, label distance=0.1mm: $x$}] {}; 
\node[main node] (u1)at (1,2) [fill=black, scale=0.5, label={above, label distance=0.1mm: $u_1$}] {};
\node[main node] (u2)at (1,1) [fill=black, scale=0.5, label={above, label distance=0.1mm: $u_2$}] {}; 
\node[main node] (u3)at (1,0) [fill=black, scale=0.5, label={above, label distance=0.1mm: $u_3$}] {};
\node[main node] (z1)at (2,2) [fill=black, scale=0.5, label={above, label distance=0.1mm: $v_1$}] {}; 
\node[main node] (z2)at (2,1) [fill=black, scale=0.5, label={above, label distance=0.1mm: $v_2$}] {}; 
\node[main node] (z3)at (2,0) [fill=black, scale=0.5, label={above, label distance=0.1mm: $v_3$}] {}; 
\node[main node] (y)at (3,1) [fill=black, scale=0.5, label={above, label distance=0.1mm: $y$}] {};

\draw (x) edge (u1) edge  (u3) edge (z2)
(u1) edge (z1) edge (z2) edge (z3)
(u2) edge (z1) edge (z3)
(u3) edge (z1) edge (z2) edge (z3)
(y) edge (z1) edge (z2) edge (z3);

\end{tikzpicture}
\hspace{0.5cm}
\vspace{0.5cm}
\begin{tikzpicture}[node distance=2.5cm,main node/.style={minimum size = .7cm,circle,fill=white!20,draw}, b node/.style={circle,fill=white!20,draw}, scale=1.5, every node/.style={scale=0.8}]   

\node[main node] (x)at (1.5,-1.5)[circle,fill=white!5, scale = 1.3] {$x$};
\node[main node] (u1)at (0,-2.5)[circle,fill=white!20]{$u_1$};
\node[main node] (u2)at (1.5,-2.5) {$u_2$};
\node[main node] (u3)at (3,-2.5) {$u_3$};
\node[main node] (v1)at (-0.5,-3.5)[circle,fill=white!20] {$v_1$};
\node[main node] (v2)at (0,-3.5)[circle,fill=white!20] {$v_2$} ;
\node[main node] (v3)at (0.5,-3.5) [circle,fill=white!20] {$v_3$};
\node[main node] (v12)at (1,-3.5)[circle,fill=white!20] {$v_1$};
\node[main node] (v22)at (1.5,-3.5)[circle,fill=white!20] {$v_2$} ;
\node[main node] (v32)at (2,-3.5) [circle,fill=white!20] {$v_3$};
\node[main node] (v13)at (2.5,-3.5)[circle,fill=white!20] {$v_1$};
\node[main node] (v23)at (3,-3.5)[circle,fill=white!20] {$v_2$} ;
\node[main node] (v33)at (3.5,-3.5) [circle,fill=white!20] {$v_3$};
\node[main node] (y1)at (-0.5,-4.5)[circle,fill=white!20] {$y$};
\node[main node] (y2)at (0,-4.5)[circle,fill=white!20] {$y$} ;
\node[main node] (y3)at (0.5,-4.5) [circle,fill=white!20] {$y$};
\node[main node] (y12)at (1,-4.5)[circle,fill=white!20] {$y$};
\node[main node] (y22)at (1.5,-4.5)[circle,fill=white!20] {$y$} ;
\node[main node] (y32)at (2,-4.5) [circle,fill=white!20] {$y$};
\node[main node] (y13)at (2.5,-4.5)[circle,fill=white!20] {$y$};
\node[main node] (y23)at (3,-4.5)[circle,fill=white!20] {$y$} ;
\node[main node] (y33)at (3.5,-4.5) [circle,fill=white!20] {$y$};

\draw (x) edge (u1) edge (u2) edge (u3)
(u1) edge (v1) edge (v2) edge (v3)
(u2) edge (v12) edge (v22) edge (v32)
(u3) edge (v13) edge (v23) edge (v33)
(v1) edge (y1) (v2) edge (y2) (v3) edge (y3)
(v12) edge (y12) (v22) edge (y22) (v32) edge (y32)
(v13) edge (y13) (v23) edge (y23) (v33) edge (y33)
;

\end{tikzpicture}
\caption{A graph and its $x$-$y$-path tree. Each vertex $w$ in the tree is labelled with $t(w)$.}
\label{treex}
\end{figure}

\begin{prop}
\label{disj}
Let $G$ be a graph containing vertices $x$ and $y$ such that $y \notin N[x]$. Let $T$ be the $x$-$y$-path tree of $G$ rooted at $v_0$. Let $P$ be a path in $T$ starting at $v_0$. If $V(P)$ sees a vertex $w \in V(G)$, then there exists a unique $u \in N_T[V(P)]$ such that $t(u) = w$. 
\end{prop}
\begin{proof}
This follows immediately from the construction of $T$: $u$ is a child in $T$ of the first vertex $v$ in $P$ such that $t(v)$ is adjacent to $w$ in $G$.
\end{proof}

The following terminology will also be used later during the proof of Theorem \ref{main} (maximising induced cycles), as well as in this section for the proof of Theorem \ref{thm:path} (maximising induced paths between a pair of vertices).
\begin{defn}\label{treestuff}
Let $G$ be a graph containing vertices $x$ and $y$ such that $y \notin N[x]$. Let $T$ be the $x$-$y$-path tree of $G$. For any $z \in V(T)$, $z$ is the child of a unique vertex $w$. Define $B(z)$, the \emph{branch rooted at} $z$, to be the component of $T \backslash \{wz\}$ containing $z$. Also, 
define $L(z)$ to be the number of leaves of $T$ that are contained in $B(z)$ and define $L_y(z)$ to be the number of leaves $l$ in $T$ contained in $B(z)$ such that $t(l) = y$. If it is unclear which tree we are considering, we will write $B_T(z)$, $L_T(z)$, etc.
\end{defn}

Observe that for $w \in T$, $t(w)=y$ only if $w$ is a leaf. It directly follows from Definition \ref{treexy} that
\begin{equation}
\label{lxy}
p_2(G;x,y) = L_y(v_0).
\end{equation} 

In order to prove Theorem \ref{thm:path}, we use the following lemma about $x$-$y$-path trees. For $w \in V(T)$, we write $D(w) := |C(w)|$.

\begin{lem}\label{pathtree}
Let $G$ be a graph on $n \ge 4$ vertices. Let $x$ and $y$ be distinct vertices in $V(G)$, with $y \notin N[x]$ and $p_2(G;x,y)>0$. Let $T$ be the $x$-$y$-path tree rooted at $v_0$ and $P:= v_0,\ldots, v_k$ be any path in $T$ where $v_k$ is a leaf. Then:

\begin{itemize}
\item [(i)] $L_y(v_0) \le f_2(n)$. 
\item [(ii)] If $L_y(v_0)=f_2(n)$, then: 
\begin{itemize}
\item [(a)] for any $v_j$ and for all $u,w \in C(v_j)$, we have $L_y(u) = L_y(w)$;
\item [(b)] $V(P) \backslash \{v_k\}$ sees every vertex of $G$ and $t(v_k) = y$.
\end{itemize}
\end{itemize}
\end{lem}

\begin{proof}
Sequentially choose a path $v_0,v_1,\ldots, v_k \subseteq V(T)$, where $v_k$ is a leaf. At vertex $v_j$ we choose $v_{j+1}$ to be some $z \in C(v_j)$ such that $L_y(z) = \max\{L_y(x): x \in C(v_j)\}$. Let $\mathcal{P}$ be the set of paths that can be obtained in this manner and fix $P:= v_0,\ldots,v_k \in \mathcal{P}$.

We first show that $t(v_k) = y$. Suppose otherwise that $t(v_k) \not= y$. As $p_2(G;x,y)>0$, by construction of $P$ we have $L_y(v_{k-1}) > 0$. If $v_{k-1}$ had a child $u$ with $t(u)=y$, by construction of $T$ this would be the only child of $v_{k-1}$. Thus $v_{k-1}$ has no such child. This fact, along with the fact $L_y(v_{k-1}) > 0 $ implies that $v_{k-1}$ has a child $z$ with $L_y(z) >0$. As $L_y(v_k) = 0$, we have $L_y(z) > L_y(v_k)$, a contradiction. 

For $0 \le i \le k-1$, we have
\begin{equation}\label{ind}
L_y(v_i) = \sum_{z \in C(v_i)} L_y(z) \le D(v_i)\max \{L_y(z): z \in C(v_i)\} = D(v_i)L_y(v_{i+1}).
\end{equation}
By repeatedly applying (\ref{ind}) we get
\begin{equation}\label{leafprod}
L_y(v_0) \le D(v_0) \max \{L_y(z): z \in C(v_0)\} \le \ldots \le L_y(v_{k-1})\prod_{i=0}^{k-2}D(v_i).
\end{equation}
As $t(v_k) = y$, $v_k$ is the only child of $v_{k-1}$ (by construction of $T$) and $L_y(v_{k-1}) = 1$. 
Thus
\begin{equation}\label{leafend}
L_y(v_0) \le \prod_{i=0}^{k-2}D(v_i),
\end{equation}
where $\sum_{i=0}^{k-2} D(v_i) \le n-2$, as $v_0,\ldots, v_{k-2}$ have disjoint sets of children in $G\backslash \{x,y\}$ by Propostion \ref{disj}. 

A quick check shows that the maximal value of $\prod_{i=0}^{k-2}D(v_i)$ subject to $\sum_{i=0}^{k-2} D(v_i) \le n-2$ occurs only in the following cases: 
\begin{itemize}
\item If $n \equiv 2$ modulo 3, we have $D(v_i) = 3$ for all $i$. 
\item If $n \equiv 0$ modulo 3, we have either $D(v_i) = 4$ for exactly one $i$ and $D(v_j) = 3$ for all $j \not= i$; or there are $i_1, i_2$ such that $D(v_i)=2$ for $i = i_1, i_2$, and $D(v_j) = 3$ for all $i \notin \{i_1,i_2\}$. 
\item If $n \equiv 1$ modulo 3, we have $D(v_i) = 2$ for exactly one $i$, and $D(v_j) = 3$ for all $j \not= i$. 
\end{itemize}
Thus we see that the maximal possible value of $\prod_{i=0}^{k-2}D(v_i)$ is $f_2(n)$, and so $L_y(v_0) \le f_2(n)$ as required for (i). 

When $L_y(v_0) = f_2(n)$ we have 
\begin{equation}
\label{pf2}
\prod_{i=0}^{k-2}D(v_i) = f_2(n).
\end{equation}
This is only possible if $\sum_{i=0}^{k-2} D(v_i) = n-2$ and the $D(v_i)$ take the values defined in the above cases. In addition, we have equality in (\ref{leafend}) and hence in (\ref{ind}) for each value of $0 \le i \le k-1$. Therefore, for each $i$ and for all $z,w \in C(v_i)$, we have $L_y(z) = L_y(w)$. 

Suppose there exists a path $X := x_0,\ldots, x_k$, where $x_0 = v_0$ and $x_k$ is a leaf, such that $X \notin \mathcal{P}$. We will derive a contradiction and hence conclude that every such path is in $\mathcal{P}$.

Choose $P':= y_0, \ldots, y_k \in \mathcal{P}$ so that it coincides with $X$ on the longest possible initial segment, i.e.,~so that $i$ is maximal such that $y_0,\ldots, y_i = x_0, \ldots, x_i$. As $X \notin \mathcal{P}$, for some $j$ we have $L_y(x_{j}) \not= L_y(y_{j})$, but $x_i = y_i$ for $i <j$. But by the argument of the previous paragraph, as $P' \in \mathcal{P}$, we have that for each $i$, $L_y(z) = L_y(w)$ for all $z,w \in C(y_i)$. Thus as $x_{j-1} = y_{j-1}$, we have $x_{j} \in C(y_{j-1})$ and $L_y(x_{j}) = L_y(y_{j})$, a contradiction. So $X \in \mathcal{P}$, as required. Again by the argument of the previous paragraph, for any $x_j$ and any $u,w \in C(x_j)$ we have $L_y(u) = L_y(w)$. This concludes part (ii a)

As $X \in \mathcal{P}$, (\ref{pf2}) holds for $X$ (our choice of $P \in \mathcal{P}$ was arbitrary). But then we have $\sum_{j=0}^{k-2} D(x_j) = n-2$ and so $X \backslash \{x_k\}$ sees every vertex of $G$ as required for (ii b). As $X \in \mathcal{P}$, $t(x_k) = y$. This concludes the proof of (ii b).
\end{proof}

We now complete the proof of Theorem \ref{thm:path}.

\begin{proof}[Proof of Theorem \ref{thm:path}]
Let $T := T(x,y)$ be the \emph{x-y}-path tree of $G$ rooted at $v_0$. By (\ref{lxy}), the number of induced paths between $x$ and $y$ is precisely the number of leaves $l \in T$ such that $t(l) = y$. As by Lemma \ref{pathtree}(i) we know $L_y(v_0) \le f_2(n)$ and moreover we know there exist graphs $H$ such that $p_2(H) = f_2(n)$ (just pick $H \in \mathcal{F}_n$), then $L_y(v_0) = f_2(n)$. We will show that $G$ is in $\mathcal{F}_n$. First we will show that $G$ is a braid.

Let $P := v_0,\ldots, v_k$ be the shortest path in $T$ such that $t(v_0) = x$ and $t(v_k) = y$. For $i \in \{0,\ldots,k-1\}$, define $C_{i+1}:= C(v_i)$ to be the set of children of $v_i$ in $T$ (note that $v_{i+1} \in C_{i+1})$. Define $V_0 := \{x\}$ and define $V_i:= t(C_{i})$ for $1 \le i \le k$. Therefore, $V_k = \{y\}$. The sets $V_i$ are disjoint by Proposition \ref{disj}. We also have that $\bigcup_{i=0}^{k} V_i = V(G)$, as $V(P)$ sees every vertex in $G$ by Lemma \ref{pathtree}(ii b).
Theorem \ref{thm:path} will follow immediately from the next claim.
\begin{claim}\label{Gbraid}
$G$ is the braid $(\{x\},V_1, \ldots, V_{k-1}, \{y\})$. 
\end{claim}
\begin{proof}
We prove by reverse induction on $j$ that the graph induced by $\bigcup_{i=j}^{k} V_i$ is a braid $(V_j, \ldots, V_{k})$ in $G$. First, note that by Lemma \ref{pathtree}(ii b), every leaf $l \in T$ satisfies $t(l) = y$. Thus no vertex in $C_{k-1}$ can be a leaf of $T$ (else we would have a shorter path to $y$ in $G$) and every vertex in $C_{k-1}$ has a child. Since, by Lemma \ref{pathtree}(ii b), $V(P)\backslash \{v_k\}$ sees every vertex of $G$, all vertices  except $y$ have been seen by $v_0,\ldots, v_{k-2}$. Thus every child $z$ of a vertex in $C_{k-1}$ satisfies $t(z) = y$, otherwise it would contradict Proposition \ref{disj}. Therefore every vertex in $C_{k-1}$ has exactly one child $z$ and $t(z) = y$. Therefore $(V_{k-1}, \{y\})$ is a braid, completing our base case. 

We now show that $N_G(y) = V_{k-1}$. Suppose there exist some $0 \le j < k-1$ and some $u \in V_j$ such that $u \in N(y)$. Then there exists a vertex $w$ in $T$ such that $w$ is a child of $v_{j-1}$, $t(w) = u$ and $w$ has a child $z$ with $t(z) = y$. The path $v_0,\ldots,v_{j-1},w,z \subseteq T$ is shorter than $P$, a contradiction. So $y$ is adjacent to no vertex in $\bigcup_{i=0}^{k-2} V_i$. As by Lemma \ref{pathtree}(ii b) $V(P)\backslash \{v_k\}$ sees every vertex of $G\backslash \{y\}$, $V(G) \subseteq \bigcup_{i=0}^{k-1} V_i$ and $N_G(y) = V_{k-1}$.

Now suppose that for $1 \le s+1 \le k-1$, the inductive hypothesis holds for $j =s+1$. So $(V_{s+1}, \ldots, V_{k})$ is a braid in $G$. We will show that $(V_{s}, \ldots, V_{k})$ is a braid in $G$. In order to do this, we need to show that for each $s+1 \le r \le k-1$, and every $u \in V_r$:
\begin{equation}
\label{braidcond}
V_{r-1} \cup V_{r+1} \subseteq N(u) \subseteq V_{r-1}\cup V_r \cup V_{r+1}.
\end{equation}

As by our inductive hypothesis $(V_{s+1}, \ldots, V_{k})$ is a braid in $G$, we know that (\ref{braidcond}) is satisfied for each $s+2 \le r \le k-1$ and every $u \in V_r$. We also know $V_{s+1} \subseteq N(u)$ for any $u \in V_{s+2}$. So it suffices to show for each $u \in V_{s+1}$, 
$$V_s \subseteq N(u) \subseteq V_{s} \cup V_{s+1} \cup V_{s+2}.$$

 We first show there are no edges between $\bigcup_{i=0}^{s-1} V_i$ and $V_{s+1}$. Suppose, for some $i \le s-1$, there exists a vertex $v \in C_i$ with a child $z$ such that $t(z) \in V_{s+1}$. Then there exists a shorter path in $T$ from $v_0$ to $v_k$: the path $v_0, \ldots, v_{i-1}, v, z, v_{s+2} \ldots, v_k$. This contradicts our choice of $P$ as the shortest such path. Thus no such vertex $v$ exists. So there are no edges between $\bigcup_{i=0}^{s-1} V_i$ and $V_{s+1}$. As $(V_{s+1},\ldots,V_k)$ is a braid in $G$, there are no edges between $\bigcup_{i=s+3}^k V_i$ and $V_{s+1}$. Thus $N(V_{s+1})\subseteq V_s \cup V_{s+1} \cup V_{s+2}$. 

It remains to show that $\{uw: u \in V_{s}, w \in V_{s+1}\} \subseteq E(G)$. Suppose there exists some $v \in C_{s}\backslash \{v_s\}$ and some $z \in C_{s+1}$ such that $t(v)$ is not adjacent to $t(z) \in V_{s+1}$. 

We know $t(v) \not= y$, it would contradict the choice of $P$ otherwise as $s < k-1$. As by Lemma \ref{pathtree} every leaf $l$ satisfies $t(l)=y$, $v$ is not a leaf and has a child $u$. We know that $t(u)$ is a neighbour of $t(v)$ in $G$ and that: 
\begin{itemize}
\item $t(u) \notin \bigcup_{i=0}^{s} V_i$, as $t(u)$ is unseen by $\{v_0, \ldots, v_{s-1}\}$ by construction of $T$;
\item $t(u) \notin \bigcup_{i=s+2}^{k} V_i$, as $V_{s+1}, \ldots, V_{k}$ forms a braid in $G$.
\end{itemize}
Thus $t(u) \in V_{s+1}$. 

If $s+2 = k$ (and so $V_{s+2} = V_k = \{y\}$) then consider the path $v_0,v_1,\ldots, v_{s-1},v \in T$. We have $L_y(v) = D(v) < D(v_s) = L_y(v_s)$, contradicting Lemma \ref{pathtree}(ii a).

Therefore $s + 2 < k$. Since every leaf $l \in T$ satisfies $t(l) = y$, by construction of $T$ any induced path $x,x_1,\ldots,x_j$ in $G$ such that $y \notin N(x_j)$ can be extended to an induced path terminating at $y$.

We consider two cases (see Figure \ref{cases12} for an illustration). 

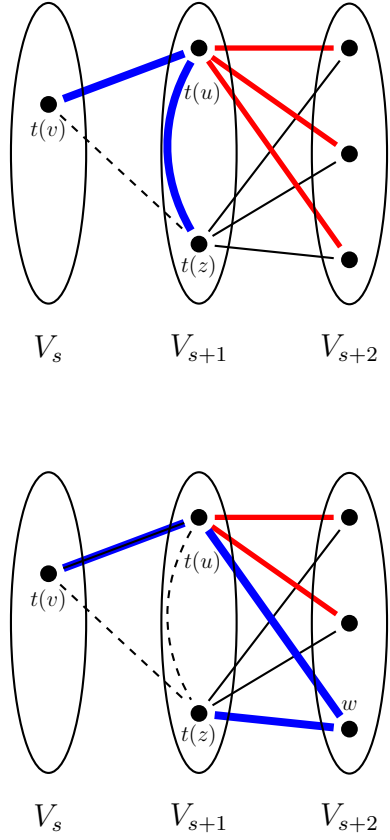
\begin{figure}[h]

\centering

\begin{tikzpicture}[thick,
  every node/.style={draw,circle}, 
  fsnode/.style={fill=black,scale=0.5},
  ssnode/.style={fill=black,scale=0.5},
  wsnode/.style={fill=black,scale=0.5},
  rsnode/.style={fill=black,scale=0.5},
  every fit/.style={ellipse,draw,inner sep=45pt,text width=0.1cm},
  -,shorten >= 3pt,shorten <= 3pt
]

\node[draw=none,fill=none] at (0,-1.4) (y) {};

% the vertices of Y
\begin{scope}[yshift = -.75cm, start chain=going below,node distance=12mm]
\foreach \i in {1}
  \node[fsnode,on chain] (f\i){};
\end{scope}

% the vertices of Z
\begin{scope}[xshift=2cm,yshift=-0cm,start chain=going below,node distance=24mm]
\foreach \i in {1,2}
  \node[ssnode,on chain] (s\i)  {};
\end{scope}

% the vertices of W
\begin{scope}[xshift=4cm,yshift=-0cm,start chain=going below,node distance=12mm]
\foreach \i in {1,2,...,3}
  \node[wsnode,on chain] (w\i)  {};
\end{scope}

\draw (f1) edge[-,line width = 3pt, color = blue] (s1)  (f1) edge[dashed] (s2) 
(w1) edge[-, line width = 2pt, color = red] (s1) (w2) edge[-] (s2)  (w1) edge[-] (s2)  (w2) edge[-, line width = 2pt, color = red] (s1)  (w3) edge[-, line width = 2pt, color = red] (s1) (w3) edge[-] (s2) (s1) edge [bend right, line width = 3pt, color=blue] (s2)

;

\node[draw=none,fill=none] at (2,-1.4) (r) {};
\node[draw=none,fill=none] at (2,-5.4) (q) {};
\draw (y) ellipse (0.5cm and 2cm);
\draw (r) ellipse (0.5cm and 2cm);
\draw (w2) ellipse (0.5cm and 2cm);

\node[draw=none,fill=none, scale = 0.7] at (0,-1.1) {$t(v)$};
\node[draw=none,fill=none, scale = 0.7] at (2.05,-0.6) {$t(u)$};
\node[draw=none,fill=none, scale = 0.7] at (2,-2.9) {$t(z)$};
\node[draw=none,fill=none] at (0,-4.0) {$V_s$};
\node[draw=none,fill=none] at (2,-4.0) {$V_{s+1}$};
\node[draw=none,fill=none] at (4,-4.0) {$V_{s+2}$};
\end{tikzpicture}

\begin{tikzpicture}[thick,
  every node/.style={draw,circle}, 
  fsnode/.style={fill=black,scale=0.5},
  ssnode/.style={fill=black,scale=0.5},
  wsnode/.style={fill=black,scale=0.5},
  rsnode/.style={fill=black,scale=0.5},
  every fit/.style={ellipse,draw,inner sep=45pt,text width=0.1cm},
  -,shorten >= 3pt,shorten <= 3pt
]

\node[draw=none,fill=none] at (0,-1.4) (y) {};
\node[draw=none,fill=none] at (2,-1.4) (z) {};
% the vertices of Y
\begin{scope}[yshift = -.75cm, start chain=going below,node distance=12mm]
\foreach \i in {1}
  \node[fsnode,on chain] (f\i){};
\end{scope}

% the vertices of Z
\begin{scope}[xshift=2cm,yshift=-0cm,start chain=going below,node distance=24mm]
\foreach \i in {1,2}
  \node[ssnode,on chain] (s\i)  {};
\end{scope}

% the vertices of W
\begin{scope}[xshift=4cm,yshift=-0cm,start chain=going below,node distance=12mm]
\foreach \i in {1,2,3}
  \node[wsnode,on chain] (w\i)  {};
\end{scope}

\draw (f1) edge[-, line width = 3pt, color= blue] (s1) (f1) edge[dashed] (s2)  edge[-] (s1) (w2) edge[-] (s2) (w1) edge[-] (s2)  (w2) edge[-, line width = 2pt, color = red] (s1)  (w3) edge[-, line width = 3pt, color= blue] (s1) (w3) edge[-, line width = 3pt, color= blue] (s2) (s1) edge [bend right, dashed] (s2) (s1) edge [-, line width = 2pt, color = red] (w1);

\draw (y) ellipse (0.5cm and 2cm);
\draw (z) ellipse (0.5cm and 2cm);
\draw (w2) ellipse (0.5cm and 2cm);

\node[draw=none,fill=none, scale = 0.7] at (0,-1.1) {$t(v)$};
\node[draw=none,fill=none, scale = 0.7] at (2.05,-0.6) {$t(u)$};
\node[draw=none,fill=none, scale = 0.7] at (2,-2.9) {$t(z)$};
\node[draw=none,fill=none,scale = 0.7] at (4,-2.5) {$w$};
\node[draw=none,fill=none] at (0,-4.0) {$V_s$};
\node[draw=none,fill=none] at (2,-4.0) {$V_{s+1}$};
\node[draw=none,fill=none] at (4,-4.0) {$V_{s+2}$};

\end{tikzpicture}

\caption{Examples of the cases we get if $t(v)$ is not adjacent to every vertex in $V_{s+1}$. The upper picture is the case $t(u)$ is adjacent to $t(z)$. The lower picture is the case $t(u)$ is not adjacent to $t(z)$. The dashed lines represent non-edges. The blue lines are the induced path we take from $t(v)$ in each case. The red edges depict which vertices have been seen by $P\backslash \{t(z)\}$.}
\label{cases12}
\end{figure}

First suppose that $t(u)$ is adjacent to $t(z)$. Consider $P:= t(v_0),\ldots, t(v_{s-1}),t(v),t(u),t(z)$, an induced path in $G$. As $t(z) \in V_{s+1}$ and $s+1 < k-1$, $y$ is not adjacent to $t(z)$ and so it is possible to extend $P$ to an induced path terminating at $y$. As $(V_{s+1},\ldots, V_k)$ is a braid in $G$, any extension of $P$ to an induced path that terminates at $y$ contains a vertex from $V_{s+2}$. However, $V_{s+2} \subseteq N_G(t(u))$ (and so $V_{s+2}$ has been seen by $V(P_u)$). It is therefore impossible to extend $P$ to an induced path terminating at $y$, a contradiction.

Now suppose that $t(u)$ is not adjacent to $t(z)$. Let $w$ be a neighbour of $t(u)$ in $V_{s+2}$. Consider the induced path $P:= t(v_0),\ldots, t(v_{s-1}),t(v),t(u),w,t(z)$. Observe that $y \notin N(t(z))$. As $t(z) \in V_{s+1}$ and $(V_{s+1},\ldots, V_k)$ is a braid in $G$, any extension of $P$ from $t(z)$ to an induced path that terminates at $y$ passes through $V_{s+2}$. However, $V_{s+2} \subseteq N_G(t(u))$ and so has been seen by $V(P_{t(u)})$. It is therefore impossible to extend this $P$ from $t(z)$ to an induced path terminating at $y$, a contradiction. 

Thus $\{uw: u \in V_{s}, w \in V_{s+1}\} \subseteq E(G)$. We conclude that the graph induced by $\bigcup_{i=s}^{k} V_i$ is indeed a braid $(V_{s}, \ldots, V_{k})$ in $G$. Claim \ref{Gbraid} now follows by induction.
\end{proof}
We have $|V_i| = D(v_{i-1})$. As $\prod_{i=0}^{k-2}D(v_i) = p_2(F(n)),$ a straightforward check shows that the braid is in $\mathcal{F}_n$. Hence Theorem \ref{thm:path} follows.
\end{proof}

\subsection{Odd and even induced paths between a pair of vertices}
 Let $G$ be a graph and let $x$ and $y$ be distinct vertices in $V(G)$. We will define similar notions for odd and even paths as we did for paths in general at the start of Section \ref{section:path}. Define $p_2^o(G;x,y)$ to be the number of induced odd paths in $G$ beginning at $x$ and ending at $y$. Also define:
 
 $$p_2^o(G):= \max \{p_2^o(G;x,y): x,y \in V(G)\},$$
 and 
 $$p_2^o(n):= \max \{p_2^o(G):|V(G)| = n\}.$$
In addition, define  $p_2^e(G;x,y)$ to be the number of induced even paths in $G$ beginning at $x$ and ending at $y$. In the even case, define $p_2^e(G)$ and $p_2^e(n)$ analogously to $p_2^o(G)$ and $p_2^o(n)$.

In this subsection we will determine the structure of the $n$-vertex graphs that contain $p_2^o(n)$ induced odd paths (or $p_2^e(n)$ induced even paths) between some pair of vertices (proving Theorems \ref{oddpath} and \ref{evenpath}). Theorem \ref{oddpath} will be used to prove Theorem \ref{oddcycle}. The extremal graphs for this path problem will have a certain structure that depends on the value of $n$ modulo 6. For $n \ge 10$, define $\mathcal{F}^o_n$ to be the set of all braids $B$ with the following properties.

\begin{itemize}
\item $|V(B)| = n$.
\item $B$ has end clusters of size 1.
\item All central clusters of $B$ have size three except:
\begin{itemize}
\item a single cluster of size 4, when $n \equiv 0$ modulo 6;
\item either two clusters of size 4 or four clusters of size 2, when $n \equiv 1$ modulo 6;
\item three clusters of size 2, when $n \equiv 2$ modulo 6;
\item two clusters of size 2, when $n \equiv 3$ modulo 6; and
\item one cluster of size 2, when $n \equiv 4$ modulo 6.
\end{itemize}
\end{itemize}

Let $F \in \mathcal{F}^o_n$ and suppose that the end clusters are $\{x\}$ and $\{y\}$. Observe that every induced path between $x$ and $y$ is odd. It is not difficult to check that for all $F \in \mathcal{F}^o_n$ we have $p_2^o(F) = p_2^o(F;x,y)$. Every graph in $\mathcal{F}^o_n$ contains the same number of induced paths between $x$ and $y$ and so we define:
$$ f_2^o(n) = \left\{ \begin{array}{c l}      
    4 \cdot 3^{(n-6)/3}  & \text{ for } n \equiv 0 \text{ modulo 6}\\
    2^4 \cdot 3^{(n-10)/3}  & \text{ for } n \equiv 1 \text{ modulo 6} \\
    2^3 \cdot 3^{(n-8)/3} & \text{ for } n \equiv 2 \text{ modulo 6}\\
    2^2 \cdot 3^{(n-6)/3}  & \text{ for } n \equiv 3 \text{ modulo 6}\\
    2 \cdot 3^{(n-4)/3}  & \text{ for } n \equiv 4 \text{ modulo 6} \\
    3^{(n-2)/3} & \text{ for } n \equiv 5 \text{ modulo 6}.\\    
    
\end{array}\right.$$

The following is a theorem for odd paths analogous to Theorem \ref{thm:path}. 

\begin{thm}\label{oddpath}
Let $G$ be a finite graph on $n \ge 10$ vertices and let $x$ and $y$ be distinct vertices of $G$. Suppose that $p_2^o(G;x,y) = p_2^o(n)$. Then $G$ is isomorphic to a graph in $\mathcal{F}^o_n$ with end clusters $\{x\}$ and $\{y\}$.
\end{thm}

The proof of this theorem is very similar to the proof of Theorem \ref{thm:path} and a sketch will be given later in this subsection.

We will also state a version of Theorem \ref{oddpath} for even length paths (Theorem \ref{evenpath}). As one would expect, the extremal graphs differ from those in the odd case. Thus for $n \ge 10$, define $\mathcal{F}^e_n$ to be the set of all braids $B$ with the following properties.

\begin{itemize}
\item $|V(B)| = n$.
\item $B$ has end clusters of size 1.
\item All central clusters of $B$ have size three except:
\begin{itemize}
\item two clusters of size 2, when $n \equiv 0$ modulo 6; 
\item one cluster of size 2, when $n \equiv 1$ modulo 6;
\item a single cluster of size 4, when $n \equiv 3$ modulo 6;
\item either two clusters of size 4 or four clusters of size 2, when $n \equiv 4$ modulo 6; and
\item three clusters of size 2, when $n \equiv 5$ modulo 6.

\end{itemize}
\end{itemize}
Observe that the extremal graphs in the odd and even cases are essentially the same (shifting by 3 modulo 6), as when $n\ge 13$ we can delete a cluster of size 3 to get from an extremal graph for the odd case to an extremal graph for the even case (or vice versa). See Figure \ref{evo} for an example.

\begin{figure}
\centering
\begin{tikzpicture}[thick,
  every node/.style={draw,circle}, 
  fsnode/.style={fill=black,scale=0.5},
  ssnode/.style={fill=black,scale=0.5},
  wsnode/.style={fill=black,scale=0.5},
  tsnode/.style={fill=black,scale=0.5},
  rsnode/.style={fill=black,scale=0.5},
  every fit/.style={ellipse,draw,inner sep=45pt,text width=0.1cm},
  ->,shorten >= 3pt,shorten <= 3pt
]

\node[draw=none,fill=none] at (0,-1.4) (y) {};

% the vertices of Y
\begin{scope}[yshift = -.75cm, start chain=going below,node distance=12mm]
\foreach \i in {1,2}
  \node[fsnode,on chain] (f\i){};
\end{scope}

% the vertices of Z
\begin{scope}[xshift=2cm,yshift=-0cm,start chain=going below,node distance=12mm]
\foreach \i in {1,2,...,3}
  \node[ssnode,on chain] (s\i)  {};
\end{scope}

% the vertices of W
\begin{scope}[xshift=4cm,yshift=-0cm,start chain=going below,node distance=12mm]
\foreach \i in {1,2,...,3}
  \node[wsnode,on chain] (w\i)  {};
\end{scope}

\begin{scope}[xshift=6cm,yshift=-0cm,start chain=going below,node distance=12mm]
\foreach \i in {1,2,...,3}
  \node[tsnode,on chain] (t\i)  {};
\end{scope}

\begin{scope}[xshift=8cm,yshift=-14mm,start chain=going below,node distance=12mm]
\foreach \i in {1}
  \node[rsnode,on chain] (r\i)  {};
\end{scope}

\begin{scope}[xshift = -2cm, yshift=-14mm,start chain=going below,node distance=12mm]
\foreach \i in {1}
  \node[rsnode,on chain] (p\i)  {};
\end{scope}

\draw (p1) edge[-] (f1)(p1) edge[-] (f2) (f1) edge[-] (s1) (f2) edge[-] (s2)  (f1) edge[-] (s2) (f1) edge[-] (s3) (f2) edge[-] (s1) (f2) edge[-] (s3) 
(w1) edge[-] (s1) (w2) edge[-] (s2) (w3) edge[-] (s3) (w1) edge[-] (s2) (w1) edge[-] (s3) (w2) edge[-] (s1) (w2) edge[-] (s3) (w3) edge[-] (s1) (w3) edge[-] (s2)
(w1) edge[-] (t1) (w1) edge[-] (t2) (w1) edge[-] (t3) (w2) edge[-] (t1) (w2) edge[-] (t2) (w3) edge[-] (t3) (w2) edge[-] (t3) (w3) edge[-] (t1) (w3) edge[-] (t2)
(t1) edge[-] (r1) (t2) edge[-] (r1) (t3) edge[-] (r1)
;

\draw (y) ellipse (0.5cm and 1.2cm);
\draw (s2) ellipse (0.5cm and 2cm);
\draw (w2) ellipse (0.5cm and 2cm);
\draw (t2) ellipse (0.5cm and 2cm);

\end{tikzpicture}
\begin{tikzpicture}[thick,
  every node/.style={draw,circle}, 
  fsnode/.style={fill=black,scale=0.5},
  ssnode/.style={fill=black,scale=0.5},
  wsnode/.style={fill=black,scale=0.5},
  rsnode/.style={fill=black,scale=0.5},
  every fit/.style={ellipse,draw,inner sep=45pt,text width=0.1cm},
  ->,shorten >= 3pt,shorten <= 3pt
]

\node[draw=none,fill=none] at (0,-1.4) (y) {};

% the vertices of Y
\begin{scope}[yshift = -.75cm, start chain=going below,node distance=12mm]
\foreach \i in {1,2}
  \node[fsnode,on chain] (f\i){};
\end{scope}

% the vertices of Z
\begin{scope}[xshift=2cm,yshift=-0cm,start chain=going below,node distance=12mm]
\foreach \i in {1,2,...,3}
  \node[ssnode,on chain] (s\i)  {};
\end{scope}

% the vertices of W
\begin{scope}[xshift=4cm,yshift=-0cm,start chain=going below,node distance=12mm]
\foreach \i in {1,2,...,3}
  \node[wsnode,on chain] (w\i)  {};
\end{scope}

\begin{scope}[xshift=6cm,yshift=-14mm,start chain=going below,node distance=12mm]
\foreach \i in {1}
  \node[rsnode,on chain] (r\i)  {};
\end{scope}

\begin{scope}[xshift = -2cm, yshift=-14mm,start chain=going below,node distance=12mm]
\foreach \i in {1}
  \node[rsnode,on chain] (p\i)  {};
\end{scope}

\draw (p1) edge[-] (f1)(p1) edge[-] (f2) (f1) edge[-] (s1) (f2) edge[-] (s2)  (f1) edge[-] (s2) (f1) edge[-] (s3) (f2) edge[-] (s1) (f2) edge[-] (s3) 
(w1) edge[-] (s1) (w2) edge[-] (s2) (w3) edge[-] (s3) (w1) edge[-] (s2) (w1) edge[-] (s3) (w2) edge[-] (s1) (w2) edge[-] (s3) (w3) edge[-] (s1) (w3) edge[-] (s2)
(w1) edge[-] (r1) (w2) edge[-] (r1) (w3) edge[-] (r1)
;

\draw (y) ellipse (0.5cm and 1.2cm);
\draw (s2) ellipse (0.5cm and 2cm);
\draw (w2) ellipse (0.5cm and 2cm);

\end{tikzpicture}
\caption{An example of a braid in $\mathcal{F}^e_{13}$ and a braid in $\mathcal{F}^o_{10}$. There may or may not be edges within the clusters.}
\label{evo}
\end{figure}
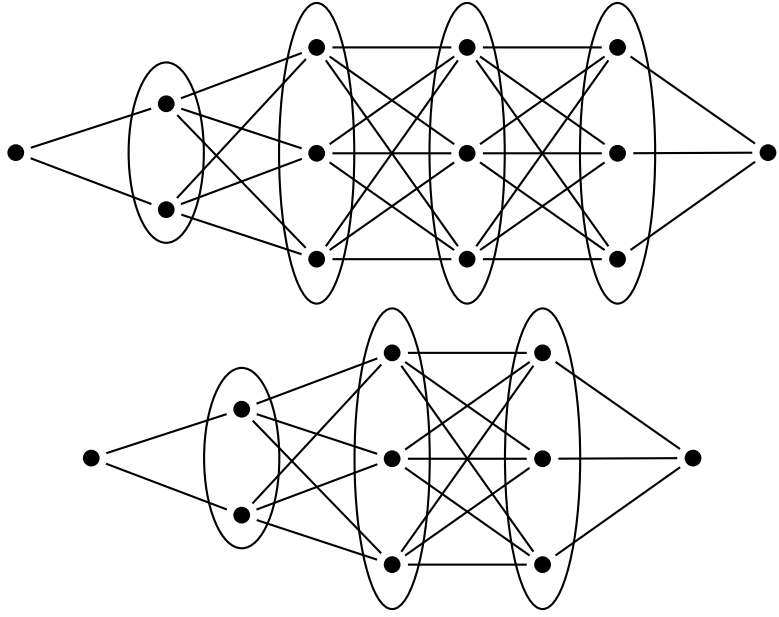

\begin{thm}\label{evenpath}
Let $G$ be a finite graph on $n \ge 10$ vertices and let $x$ and $y$ be distinct vertices of $G$. Suppose that $p_2^e(G;x,y) = p_2^e(n)$. Then $G$ is isomorphic to a graph in $\mathcal{F}^e_n$ with end clusters $\{x\}$ and $\{y\}$.
\end{thm}
To obtain the proof of Theorem \ref{oddcycle}, we only need Theorem \ref{oddpath}. We remark that the proof of Theorem \ref{oddpath} can easily be adapted to prove Theorem \ref{evenpath}, so we omit the proof of Theorem \ref{evenpath}. 

\begin{proof}[Sketch proof of Theorem \ref{oddpath}]
 Fix $x$ and $y$ to be distinct vertices of $G$ with $y \notin N[x]$. Let $T$ be the $x$-$y$-path tree rooted at $v_0$. For $z \in V(T)$ define $L_o(z)$ to be the number of leaves $l$ contained in $B(z)$ such that $t(l) = y$ and the path from $v_0$ to $l$ is odd. 

We first prove an \emph{odd-path version} of Lemma \ref{pathtree}.

\begin{claim}[Odd-path version of Lemma \ref{pathtree}]\label{clam}
Let $G$ be a graph on $n \ge 10$ vertices. Let $x$ and $y$ be distinct vertices in $V(G)$ with $y \notin N[x]$ and $p_2^o(G;x,y) >0$. Let $T$ be the $x-y$-path tree rooted at $v_0$ and $P := x_1,\ldots,x_k$ be any path in $T$ where $x_0 = v_0$ and $x_k$ is a leaf. Then:
\begin{itemize}
\item [(i)]  $L_o(v_0) \le f_2^o(n)$. 
\item [(ii)] If $L_o(v_0) = f_2^o(n)$, then:
\begin{itemize}
\item [(a)] for any $x_j$ and for all $u,w \in C(x_{j})$, we have $L_o(u) = L_o(w)$;
\item [(b)] $k$ is even;
\item [(c)] $V(P)\backslash \{x_k\}$ sees every vertex of $G$ and $t(x_k) = y$.
\end{itemize}
\end{itemize}
\end{claim}

\begin{proof}[Proof of Claim]
We essentially mimic the proof of Lemma \ref{pathtree}, replacing $\mathcal{F}_n$ with $\mathcal{F}^o_n$ and $L_y(z)$ with $L_o(z)$ for any $z \in T$. 

Sequentially choose a path $v_0,\ldots,v_k \subseteq V(T)$, where $v_k$ is a leaf. At vertex $v_j$ we choose $v_{j+1}$ to be some $z \in C(v_j)$ such that $L_o(z) = \max \{L_o(x):x \in C(v_j)\}$. Let $\mathcal{P}$ be the set of paths that can be obtained in this manner and fix $P:= v_0,\ldots,v_k \in \mathcal{P}$.

We now show that $k$ is even and $t(v_k) = y$. Suppose first that $k$ is odd. Thus any path from $v_0$ to some leaf neighbour of $v_{k-1}$ is even. As $p^o_2(G;x,y)>0$, by construction of $P$ we have $L_o(v_{k-1})>0$. So $v_{k-1}$ has some non-leaf child $z$ with $L_o(z)>0$. But then $L_o(z)> L_o(v_k) = 0$, a contradiction. So $k$ is even. The fact that $t(v_k) = y$ follows from exactly the same argument as in Lemma \ref{pathtree}.

Arguing as in (\ref{leafprod}), we see that 
\begin{equation}
L_o(v_0) \le \prod_{i=0}^{k-2}D(v_i),
\end{equation}
where $k$ is even and $\sum_{i=0}^{k-2} D(v_i) \le n-2$, as $v_0,\ldots, v_{k-2}$ have disjoint sets of children in $G\backslash \{x,y\}$ by proposition \ref{disj}. 

It is not difficult to check that the maximal value of $\prod_{i=0}^{k-2}D(v_i)$ subject to $\sum_{i=0}^{k-2} D(v_i) \le n-2$, where $k$ is even is $f_2^o(n)$. This concludes the proof of (i).

When $L_o(v_0) = f_2^o(n)$ we have for even $k$:
$$\prod_{i=0}^{k-2}D(v_i) = f_2^o(n).$$
Statement (ii a) follows from an analogous argument to the proof of (ii a) in Lemma \ref{pathtree}. Also (using an identical argument to the one used in Lemma \ref{pathtree}) we have that any path $X:= x_0,\ldots, x_j$, where $x_0 = v_0$ and $x_k$ is a leaf, is in $\mathcal{P}$. Thus $j$ is even, as required for (ii b), and $t(v_k) = y$. The other claim in statement (ii c) follows an analogous argument to the proof of (ii b) in Lemma \ref{pathtree}. This completes the proof of the claim.
\end{proof}

 In particular, we know that any path $v_0 \ldots v_k$, where $v_k$ is a leaf, is odd and satisfies $t(v_k) = y$. We now use analogous arguments to those used in the proof of Theorem \ref{thm:path} replacing $\mathcal{F}_n$ with $\mathcal{F}^o_n$, replacing $L_y(z)$ with $L_o(z)$ for any $z \in T$ and applying Claim \ref{clam} in the place of Lemma \ref{pathtree}, to show that $G$ is a braid in $\mathcal{F}^o_n$. 

\end{proof}

\section{Proof of Theorem \ref{main}}\label{section:cycle}

We fix a large constant $n_0$ and let $G_{max}$ be a graph on $n \ge n_0$ vertices, that contains $m(n)$ induced cycles. In what follows we will take $n_0$ (and thus $n$) to be sufficiently large when required and we will make no attempts to optimise the constants in our arguments.  We will show that the graph $G_{max}$ is isomorphic to $H_n$. As it turns out, Theorem \ref{stab} (the stability result) will follow almost immediately from the arguments required for the proof of Theorem \ref{main}. Therefore, in this section several lemmas are proved in more generality than is needed for the proof of Theorem \ref{main}: they will be used in their more general form in the next section. 

Given a graph $H$, let $f(H)$ denote the number of induced cycles in $H$ and for a vertex $v \in H$, let $f_v(H)$ denote the number of induced cycles in $H$ containing $v$. Observe that we have:

\begin{equation}\label{leastB}
f(G_{max}) = m(n) \ge f(H_n) \ge \begin{cases}
    3^{n/3}& \text{if } n \equiv 0 \text{ modulo } 3\\
    4\cdot3^{(n-4)/3}& \text{if } n \equiv 1 \text{ modulo } 3\\
    2\cdot3^{(n-2)/3}& \text{if } n \equiv 2 \text{ modulo } 3.\\
\end{cases}
\end{equation}

Any $G_{max}$ is connected (if it were disconnected we could add edges between two components to increase the number of induced cycles).  We begin by proving several lemmas which determine information about the structure of $G_{max}$.

\begin{lem}\label{upperv}
Let $F$ be an $n$-vertex graph. For $v \in V(F)$, we have $f_v(F) \le \binom{d(v)}{2}3^{(n- d(v) -1)/3}$.
\end{lem}
\begin{proof}
Each induced cycle containing $v$ contains exactly two neighbours of $v$. Fix a pair of vertices $u,w \in N(v)$. By Corollary \ref{pathcor} there are at most $3^{(n-d(v) - 1)/3}$ induced paths between $u$ and $w$ in $(F\backslash N[v])\cup \{u,w\}$. Thus there can be at most $3^{(n-d(v) - 1)/3}$ induced cycles in $F$ containing $\{v,u,w\}$. As there are $\binom{d(v)}{2}$ distinct pairs of neighbours of $v$, we have,
$$f_v(F) \le \binom{d(v)}{2}3^{(n-d(v)-1)/3}$$
as required.
\end{proof}

The next lemma tells us that any vertex in $G_{max}$ is contained in a constant proportion of $f(G_{max})$ induced cycles. 
\begin{lem}\label{vertex}
Let $0< c \le 1$ and let $\alpha = 0.11$. Let $F$ be an $n$-vertex graph with $f(F) \ge c\cdot m(n)$. Then:
\begin{itemize}
\item[(i)] $(1 - o(1))f(F)$ induced cycles in $F$ have length at least $\alpha n$.
\item [(ii)] $F$ contains a vertex $v$ such that $f_v(F) \ge \frac{c}{10}m(n)$.
\item [(iii)]  Every vertex $w \in V(G_{max})$ satisfies $f_w(G_{max}) \ge \frac{c}{20}m(n)$.
\end{itemize}
\end{lem}

\begin{proof}
$F$ contains at most $\sum_{i = 1}^{\lfloor \alpha n \rfloor} \binom{n}{i}$ induced cycles of length at most $\alpha n$ (for any $W \subseteq V(F)$, there exists at most one induced cycle $C$ such that $V(C) = W$). Using Stirling's approximation, we get 

$$\sum_{i = 1}^{\lfloor \alpha n \rfloor} \binom{n}{i} \le \alpha n\cdot \binom{n}{\alpha n} \le \left(1 + o(1)\right) \frac{\sqrt{\alpha n}}{\sqrt{2 \pi (1 - \alpha)}} \left[\frac{1}{\alpha^{\alpha}(1-\alpha)^{1- \alpha}}\right]^n. $$
As
$$ \frac{1}{\alpha^{\alpha}(1-\alpha)^{1- \alpha}} < 3^{1/3},$$ 
we get 
$$ \sum_{i = 1}^{\lfloor \alpha n \rfloor} \binom{n}{i} = o\left(3 ^{n/3}\right).$$
By (\ref{leastB}), $f(F) = \Omega(3^{n/3})$, so $(1-o(1))f(F)$ induced cycles in $F$ have length at least $\alpha n$, as required for (i).

Provided $n_0$ is sufficiently large, we have for all $n > n_0$,
$$\sum_{i = 1}^{\lfloor \alpha n \rfloor} \binom{n}{i} < \frac{c}{1000}\cdot 3^{n/3} < \frac{c}{100}m(n),$$
where the second inequality follows from (\ref{leastB}). Thus there exists a vertex $v$ such that 
$$f_v(F) \ge \frac{99 \alpha c}{100}m(n) \ge \frac{c}{10}m(n),$$
proving (ii).

Now suppose that there exists some vertex $w \in V(G_{max})$ with $f_w(G_{max}) < \frac{c}{20}m(n)$. Consider the graph $G'$ obtained from $G_{max}$ by duplicating the vertex $v$ and removing the vertex $w$. We have that
$$f(G') \ge f(G_{max}) + f_v(G_{max}) - 2f_w(G_{max}) > f(G_{max}),$$
a contradiction. This proves (iii).
\end{proof}

Now consider the graph obtained from $G_{max}$ by duplicating a vertex. By applying Lemma \ref{vertex}, we have
\begin{equation}\label{nvsn+1}
m(n+1) \ge \left(1 + \frac{1}{20}\right)m(n).
\end{equation}

We now use this to show that the graph $G_{max}$ has maximum degree bounded by a constant. 

\begin{lem}\label{deg}
$\Delta(G_{max}) < 30$. 
\end{lem}
\begin{proof}
Let $v$ be a vertex of maximal degree. Given $v$, split the induced cycles in $G_{max}$ into those contained in $G_{max} \backslash \{v\}$, and those containing $v$. Using Lemma \ref{upperv} we have
$$m(n) \le m(n-1) +\binom{d(v)}{2}3^{(n-d(v)-1)/3}.$$
Using (\ref{nvsn+1}) to bound $m(n-1)$ gives,
$$m(n) \le m(n)\left(1 + \frac{1}{20}\right)^{-1} +\binom{d(v)}{2}3^{(n-d(v)-1)/3}.$$
This expression rearranges to give
$$m(n) \le 21\binom{d(v)}{2}3^{(n -d(v) - 1)/3}.$$
For $d(v) \ge 30$, this implies $m(n) < 3^{(n-6)/3},$ a contradiction (as $m(n) \ge f(H_n) > 3^{(n-6)/3}$).
\end{proof}

Combining Lemma \ref{upperv} with Lemma \ref{deg} shows, for any $v \in V(G_{max})$, we have $f_v(G_{max}) \le \binom{30}{2}3^{(n-3)/3}$. By Lemma \ref{vertex} (i), we know $(1 - o(1))m(n)$ induced cycles in $G_{max}$ have length at least $0.11 n$. Thus 
$$ (1 - o(1))m(n) \le \frac{1}{0.11 n} \sum_{v \in V(G)}f_v(G_{max}) = O(3^{n/3}).$$
This along with (\ref{leastB}) implies 
\begin{equation}\label{mn}
m(n) = \Theta(3^{n/3}).
\end{equation} 

The next stage of our proof involves showing that all but a constant number of vertices $v$ in our graph have the property that their closed second neighbourhood $N^2[v]$ has the same local structure as the closed second neighbourhood of a vertex in $H_n$.  We introduce some preliminary definitions.

Given a graph $F$ and a set $S \subseteq V(F)$, we say that a vertex $v \in V(F)$ is \emph{seen by} $S$ if $v \in N[S]$ and $v$ is \emph{unseen by} $S$ otherwise. Given a subgraph $H \subseteq F$, we say $v$ is seen by $H$ if $v \in N[H]$. When it is clear which set/subgraph we are referring to, we will just say $v$ is (un)seen. For a vertex $v \in V(G)$, let $N^i(v)$ be the set of points at distance exactly $i$ from $v$. Define $N^k[v]$ to be the set of points within distance $k$ of $v$ (for example $N^3[v] = \{v\} \cup \bigcup_{i = 1}^3 N^i(v)$ and $N^1[v]:= N[v]$). 

In order to determine what sort of local structure a `typical' vertex in $G_{max}$ should have, we define a game on $F$.

\begin{defn}\label{typG}
Let $F$ be a finite graph, let $v \in V(F)$ and let $w \in V(F)\backslash N^4[v]$. We define the $w$-\emph{typical-game} on $(F,v)$ as follows. There are two players, \emph{Adversary} and \emph{Builder}. The game starts at vertex $u_1 := v$ and the players choose a sequence of vertices $\{u_2, \ldots, u_k\}$ under the following set of rules. At vertex $u_i$:
\begin{itemize}
\item If $u_i \in N^4[w]$, then Adversary is the \emph{active player}, otherwise Builder is.
\item The active player chooses a neighbour $u_{i+1}$ of $u_i$ that is unseen by $\{u_1, \ldots, u_{i-1}\}$.
\item The game terminates when a vertex $u_j$ is chosen such that $u_j$ has no neighbours unseen by $\{u_1, \ldots, u_{j-1}\}$.
\item If, for some $j$, the vertex $u_j$ does not have exactly 3 neighbours unseen by $\{u_1, \ldots, u_{j-1}\}$, we call $u_j$ \emph{bad}; we call $u_j$ \emph{good} otherwise.
\item Adversary wins if either:
\begin{itemize}
\item for some $j$, the vertex $u_j$ is in $N^4[w]$ and is bad; or 
\item upon termination of the game at vertex $u_k$, there exists a vertex in $N^4[w]$ that is unseen by $\{u_1,\ldots, u_k\}$. 
\end{itemize}
Builder wins otherwise. 
\end{itemize} 

\end{defn}

A vertex $w\in V(F) \backslash N^4[v]$ is \emph{v-typical} in $F$ if there exists a winning strategy for Builder in the $w$-typical-game on $(F,v)$. A vertex is \emph{v-atypical} otherwise. When it is clear which vertex has been chosen to play the role of $v$, we simply say that $w$ is (a)typical. Note that the set of vertices $\{u_1,\ldots, u_k\}$ chosen during the game induces a path in $F$. Also, if we play this game on $H_n$ starting at any vertex, most of the chosen vertices are good. 

The next lemma shows a $v$-typical vertex has the required local structure (see Figure~\ref{typfig}).

\begin{lem}\label{structure}
Let $F$ be a graph and let $v$ be a vertex in $F$. Suppose that $z =z_1 \in V(F) \backslash N^4[v]$ is a $v$-typical vertex. Then there exist disjoint sets of vertices $Z:= \{z_1,z_2,z_3\}, V:= \{v_1,v_2,v_3\},$ and $W:= \{w_1,w_2,w_3\}$ such that:
\begin{itemize}
\item [(i)] for $i,j \in \{1,2,3\}$, we have $N(v_i) \cap N(w_j) = Z$;
\item [(ii)] for all $i$, we have $V \cup W \subseteq N(z_i) \subseteq V \cup W \cup Z$; 
\item [(iii)] there are no edges between $V$ and $W$.
\end{itemize}
\end{lem}

\begin{figure}[h]
\centering 

\begin{tikzpicture}[thick,
  every node/.style={draw,circle}, 
  fsnode/.style={fill=black,scale=0.5},
  ssnode/.style={fill=black,scale=0.5},
  wsnode/.style={fill=black,scale=0.5},
  every fit/.style={ellipse,draw,inner sep=45pt,text width=0.1cm},
  ->,shorten >= 3pt,shorten <= 3pt
]

% the vertices of Y
\begin{scope}[start chain=going below,node distance=12mm]
\foreach \i in {1,2,...,3}
  \node[fsnode,on chain] (f\i) [label={above, label distance=0.1mm: $v_{\i}$}] {};
\end{scope}

% the vertices of Z
\begin{scope}[xshift=2cm,yshift=-0cm,start chain=going below,node distance=12mm]
\foreach \i in {1,2,...,3}
  \node[ssnode,on chain] (s\i) [label={above, label distance=0.1mm: $z_{\i}$}]  {};
\end{scope}

% the vertices of W
\begin{scope}[xshift=4cm,yshift=-0cm,start chain=going below,node distance=12mm]
\foreach \i in {1,2,...,3}
  \node[wsnode,on chain] (w\i) [label={above, label distance=0.1mm: $w_{\i}$}] {};
\end{scope}

\draw (f1) edge[-] (s1) (f2) edge[-] (s2) (f3) edge[-] (s3) (f1) edge[-] (s2) (f1) edge[-] (s3) (f2) edge[-] (s1) (f2) edge[-] (s3) (f3) edge[-] (s1) (f3) edge[-] (s2)
(w1) edge[-] (s1) (w2) edge[-] (s2) (w3) edge[-] (s3) (w1) edge[-] (s2) (w1) edge[-] (s3) (w2) edge[-] (s1) (w2) edge[-] (s3) (w3) edge[-] (s1) (w3) edge[-] (s2);

\draw (f2) ellipse (0.5cm and 2.5cm);
\draw (s2) ellipse (0.5cm and 2.5cm);
\draw (w2) ellipse (0.5cm and 2.5cm);

\node[draw=none,fill=none] at (0,-4.3) {$V$};
\node[draw=none,fill=none] at (2,-4.3) {$Z$};
\node[draw=none,fill=none] at (4,-4.3) {$W$};
\end{tikzpicture}
\caption{The local structure around a $v$-typical vertex $z_1$. Note that we have not yet determined the edges within the sets $V$, $Z$ and $W$.}\label{typfig}
 \end{figure}
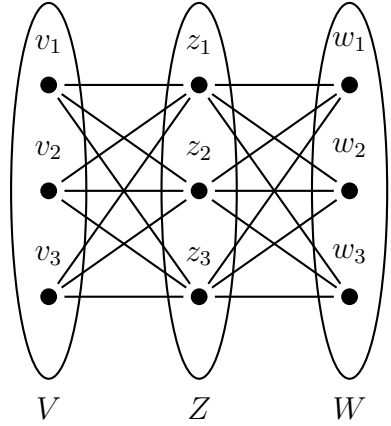

\begin{proof}
We play as Adversary in the $z$-typical-game on $(F,v)$. As $z$ is $v$-typical, Builder has a winning strategy $\sigma$. We assume that Builder uses strategy $\sigma$, and deduce information about the structure of $F$ from the results of our choices of vertices as Adversary (we know we cannot win so whatever choices we make have certain consequences). 
For each vertex $u_i$ that is chosen, let $P_{u_i}$ denote the subgraph induced by $\{u_1,\ldots, u_{i}\}$, where $u_1 = v$. So $P_{u_i}$ is an induced path between $v$ and $u_i$.

Suppose that $u_k$ is the first vertex chosen such that $u_k \in N^4(z)$ (as $z$ is typical, at some point such a vertex will be chosen). We arbitrarily choose the next two vertices $u_{k+1} \in N^3(z) \cap N(u_k)$ and $u_{k+2} \in N^2(z) \cap N(u_{k+1})$. Let $x := u_{k+2}$. This vertex is unseen by $P_{u_{k}}$ as $u_{k+1}$ was the first vertex we chose in $N^3(z)$. We also have, by choice of $u_k$, that $x \notin N(z)$. As $z$ is typical, $x$ has 3 neighbours $V := \{v_1,v_2,v_3\}$ unseen by $P_{u_{k+1}}$. As $x \in N^2(z)$, for some $i$ we have $v_i \in N(x) \cap N(z)$. Without loss of generality suppose $v_1 \in N(x) \cap N(z)$. Since we could choose $u_{k+3}$ to be $v_1$, the vertex $v_1$ has 3 neighbours unseen by $P_x$, one of which is $z$, so let the set of neighbours of $v_1$ unseen by $P_x$ be $Z:= \{z_1,z_2,z_3\}$. 

 If we choose $u_{k+3}:= v_1$ and then $u_{k+4}:= z$, we have that $z$ has 3 neighbours unseen by $P_{v_1}$. Thus $z$ has at most 5 neighbours unseen by $P_{x}$ (as $z$ could be adjacent to $z_2$ or $z_3$).

We now prove that $N(x) \cap N(z) = V$. Suppose otherwise, so without loss of generality we have $v_3 \notin N(z)$. Now we describe the set of choices we make for the remainder of the game (recall that Builder always plays by strategy $\sigma$). Choose $u_{k+3}:= v_3$. Now consider a later step in the game, but before $z$ has been chosen, and suppose the most recently chosen vertex is $u_i \in N^4[z]$, where $i \ge k+3$. Then:

\begin{itemize}
\item [(1)] If there is no vertex in $N(u_i) \cap N(z)$ that is unseen by $P_{u_{i-1}}$, choose $u_{i+1}$ arbitrarily. 
\item [(2)] If, for some $r$, the vertex $z$ has $r$ neighbours unseen by $P_{u_{i-1}}$ (by the argument above, $r \le 5$) and $u_i$ is adjacent to $j\ge 1$ of these neighbours: 
\begin{itemize}
\item [(i)] If $r-j < 3$, choose $u_{i+1} \in N(u_i) \cap N(z)$ and $u_{i+2}:=z$. As $z$ now has at most 2 neighbours unseen by $P_{u_{i+1}}$ we reach a contradiction. 
\item [(ii)] If $r-j \ge 3$, then $j \le 2$ and $u_i$ has an unseen neighbour $s$ that is not adjacent to $z$ (otherwise $u_i$ is bad, which contradicts the fact that Builder has a winning strategy). Choose $u_{i+1}:= s$. Observe that the vertex $z$ is not seen by $P_{s}$.
\end{itemize}
\end{itemize}
Once we have chosen $z$, we play arbitrarily. 

We now analyse the results of making these choices. As $z$ is typical, we will at some point enter case (2). If we are in case (2ii), we pick $u_{i+1}$ and the number of neighbours unseen by $z$ decreases, so eventually we enter case (2i) where we reach a contradiction. Thus $N(x) \cap N(z) = V$, as required.

We now know that $N(x) \cap N(z) = V$. As in the first part, if we choose $u_{k+3}:= v_1$ and $u_{k+4}:= z$, the vertex $z$ has 3 neighbours unseen by $P_{v_1}$. Call these neighbours $W:=\{w_1,w_2,w_3\}$. Thus $N(z) \supseteq V \cup W$ as required. Observe that by choice of $W$, there are no edges between $V$ and $W$, thus proving the third statement of the lemma. From the argument above we have $N(z) \subseteq V \cup W \cup Z$.

Let $w_j \in W$, for $j \in \{1,2,3\}$. We will first show that $N(v_1)\cap N(w_j) = Z$. Again we play as Adversary in a new $z$-typical-game on $(F,v)$ and assume that Builder uses the same winning strategy $\sigma$ as above. 

By playing the same strategy as in the game above, $u_{k+1}$ is the first vertex chosen in $N^3(z)$ and that the first vertex chosen in $N^2(z)$ is $x$. We know that $N(x) \cap N(z) = V$. Now we choose $u_{k+3}:= v_1$. Let $y$ be the first vertex chosen in $N^2(w_j)$. Observe that as $w_j \in N(z)$, we have $N^2(w_j) \subseteq N^3[z]$. So $y \in \{u_{k+1},x,v_1\}$. We will show that $y = v_1$.

By repeating the same argument as above (where we showed $N(x) \cap N(z) = V$) with $y$ in place of $x$ and $w_j$ in place of $z$, we see that $y$ and $w_j$ have exactly 3 common neighbours. The assumptions we needed to apply the argument above will hold here: we require that $y$ is the first vertex chosen in $N^2(w_j)$ and that we are able to make moves at vertices in $N^2[w_j]$. The latter holds as $N^2[w_j] \subseteq N^4[z]$ and, in the $z$-typical game on $(F,v)$, Adversary makes moves at vertices in $N^4[z]$. 

Suppose that $y = u_{k+1}$. Then $|N(u_{k+1}) \cap N(w_j)| = 3$ and in particular $w_j$ is adjacent to $x$. This is a contradiction as $w_j \notin V$. Now suppose that $y = x$. This implies that $w_j$ is adjacent to $v_1$. As $w_j \notin Z$, this is a contradiction. Therefore $y = v_1$ and $N(v_1) \cap N(w_j) = Z$ for each $j \in \{1,2,3\}$ (by the analogous argument to where we show $N(x) \cap N(z) = V$, with $v_1$ in place of $x$, $w_j$ in place of $z$ and $Z$ in place of $V$).

We now show that, for $i,j \in \{2,3\}$, each $v_i$ is adjacent to each $z_j$. Suppose that $v_i$ is not adjacent to $z_j$ for some $i,j \in \{2,3\}$. Again we play as Adversary in a new $z$-typical-game on $(F,v)$. We assume that Builder uses the same winning strategy $\sigma$ as above. By playing the same strategy as in both games above, the first vertex chosen in $N^2[z]$ is $x$ and $x = u_{k+2}$ in the sequence of vertices chosen. It is deduced from analogous arguments to those above that the only neighbours of $z_j$ that are unseen by $P_x$ are contained in $V \cup Z \cup W$. We know that $N(x) \cap N(z) = V$. Now choose $u_{k+3} := v_i$ and $u_{k+4}:= z$. If $z$ is adjacent to $z_j$, pick $u_{k+5}:= z_j$. Otherwise, pick $u_{k+5} := w_1$ and $u_{k+6} := z_j$. In both cases, all vertices of $W \cup V$ have been seen when $z_j$ is chosen. Therefore $z_j$ has at most one neighbour (the vertex in $Z \backslash \{z,z_j\}$) unseen by $P_{z_j}\backslash \{z_j\}$. This contradicts $z$ being typical. Thus we have $N(v_i) \cap N(w_j) = Z$, completing the proof of (i).

It remains to show that $N(z_i) \subseteq V \cup W \cup Z$ for $i \in \{2,3\}$. Suppose otherwise, that $z_i$ has a neighbour $u \notin V \cup Z \cup W$. By the above arguments, $u$ is seen by $P_x$. $u$ is not a neighbour of $x$, and as $u \in N^3[z]$, $u$ is either a neighbour of $u_k$ or $u_{k+1}$. Now consider a new $z$-typical game on $(F,v)$. Builder still plays by the winning strategy $\sigma$. This time, we play as before until we are the first vertex $u'$ which is a neighbour of $u$ (so $u'$ is either $u_k$ or $u_{k+1}$). At $u'$, we choose $u$ and then $z_i$. Observe that $N(z_i)\supseteq W \cup V$. By the arguments above, no vertices in $W \cup V$ are seen by $P_{u'}$. As $z_i$ is typical it has exactly 3 neighbours unseen by $P_{u}$, this implies that at least 3 vertices of $W \cup V$ are neighbours of $u$. However, this means that $u$ has 4 neighbours unseen by $P_{u'}$ (these 3 and $z_i$), contradicting $z$ being typical. Thus  for $i \in \{2,3\}$, $N(z_i) \subseteq V \cup W \cup Z$. This completes the proof of the lemma.

\end{proof}

We will now show that, for any $v \in V(G_{max})$, all but a bounded number of vertices in $G_{max}$ are $v$-typical. We do this in the following manner. For each vertex $v$ in $G_{max}$ we define a tree $T(v)$ that will `explore' the graph $G_{max}$ outwards from $v$. As we will see, leaves on this tree will correspond to induced paths or cycles in $G_{max}$ containing $v$. Every vertex on $T$ represents some vertex in $G_{max}$ (and many vertices in $T$ may represent the same vertex of $G_{max}$). Our proof proceeds by showing that $T$ has a particular structure, which in turn implies conditions on the structure of $G_{max}$. 

The next definition contains similar concepts to those introduced in the definition of an $x$-$y$-path-tree in Section \ref{section:path}. The main difference is that previously the vertex $y$ played a special role in the creation of leaves of our tree. Now there is no such significant vertex. 

\begin{defn}\label{etree}
For $F$ a finite graph and $v \in V(F)$, the \emph{exploration tree from v} is a tree $T = T(v)$ together with a function $t:V(T) \rightarrow V(F)$ defined as follows.
\begin{itemize}
 \item $T$ is a tree with vertex set $V(T)$ disjoint from $V(F)$. 

 \item The vertices in $T$ correspond to sets $S:= \{v,v_1,\ldots,v_j\} \subseteq V(F)$ such that $F[S]$ is an induced path $v,v_1,\ldots,v_j$ or induced cycle $v,v_1,\ldots,v_j,v$. For each $S$ where $F[S]$ is an induced path, define a vertex $w_S \in T$ and set $t(w_S) = v_j$. For each $S$ where $F[S]$ is an induced cycle (where an edge is not considered to be a cycle), define two vertices $w_S^1,w_S^2 \in T$ and set $t(w_S^1) = v_1$ and $t(w_S^2) = v_j$. We call $S$ the \emph{F-set} of $w$ (or $F$-set of $w_S^1$ and $w_S^2$, if $F[S]$ is a cycle). These vertices are the only vertices in $T$. Define the root of the tree to be $v_0:=w_{x}$.
 \item Given a vertex $w \in V(T)$ with $F$-set $\{v,v_1,\ldots,v_j\}$ we define $C(w)$, the \emph{children} of $w$, to be the set of vertices in $T$ whose $F$-set is $\{v,v_1,\ldots,v_j,z\}$ for some $z \in V(F)$. Define $N_T(v_0):= C(v_0)$. For $w \in V(T)\backslash \{v_0\}$ define $N_T(w):= C(w) \cup \{u\}$, where $u$ is the unique vertex in $T$ with $F$-set $\{v,v_1,\ldots,v_{j-1}\}$. (So two vertices are adjacent in $T$ precisely when one of their $F$-sets extends the other by one vertex. A vertex whose $F$-set induces a cycle in $F$ will be a leaf of $T$. ) 
\end{itemize}
We write $t(S):= \{t(x): x \in S\}$ for any subset $S \subseteq V(T)$ and $t(H):= \{t(x): x \in V(H)\}$ for any subgraph $H \subseteq T$. Given a set $P \subseteq V(T)$ we say that it \emph{sees} a vertex $w \in V(F)$ if $w \in N_F[t(P)]$. If $w \notin N_F[t(P)]$ we say $w$ is \emph{unseen} by $P$. Note that if some set $P$ sees $w$ then there exists $u \in N_T[P]$ such that $t(u) = w$. 
\end{defn}

As in Definition \ref{treestuff}, define a \emph{branch} and $L(u)$ for $u \in V(T)$ with respect to this tree. We now describe a correspondence between certain leaves on $T$ and induced cycles in $F$. For $z \in T$, let $L(z)$ be the number of leaves of $T$ contained in $B(z)$.

See Figure \ref{extree} for an example of an exploration tree.

\begin{figure}[H]
\centering

 \begin{tikzpicture}[node distance=2.5cm,main node/.style={minimum size = .7cm,circle,fill=white!20,draw}, b node/.style={circle,fill=white!20,draw}, scale=1.5, every node/.style={scale=0.8}]   
 
 \node[main node] (x)at (-1,0) [fill=black, scale=0.5, label={above, label distance=0.1mm: $x$}] {}; 
 \node[main node] (y)at (1,0) [fill=black, scale=0.5, label={above, label distance=0.1mm: $y$}] {}; 
 \node[main node] (v1)at (-1,-1) [fill=black, scale=0.5, label={below, label distance=0.1mm: $v_1$}] {};
 \node[main node] (z1)at (1,-1) [fill=black, scale=0.5, label={below, label distance=0.1mm: $z_1$}] {}; 
 \node[main node] (v2)at (-2,-2) [fill=black, scale=0.5, label={left, label distance=0.1mm: $v_2$}] {};
 \node[main node] (z2)at (2,-2) [fill=black, scale=0.5, label={right, label distance=0.1mm: $z_2$}] {}; 

 \draw (x) edge (y) (x) edge (v1) (x) edge (v2)
 
 (z1) edge (v1) (z1) edge (v2) (z1) edge (y)
 
 (z2) edge (y) (z2) edge (v1) (z2) edge (v2);

 \node[main node] (a1)at (0,-3)[circle,fill=white!20, minimum size=.9cm] {$v_1$};
 \node[main node] (b1)at (-3,-4)[circle,fill=white!20, minimum size=.9cm]{$z_1$};
 \node[main node] (b2)at (0,-4)[circle,fill=white!20, minimum size=.9cm] {$z_2$};
 \node[main node] (x)at (3,-4) [circle,fill=white!20, minimum size=.9cm]{$x$};
 
 \node[main node] (z11)at (-3.5,-5) [circle,fill=white!20, minimum size=.9cm] {$y$};
 \node[main node] (z21)at (-2.5,-5)[circle,fill=white!20, minimum size=.9cm] {$v_2$};
 \node[main node] (z12)at (-0.5,-5) [circle,fill=white!20, minimum size=.9cm] {$y$};
 \node[main node] (z22)at (0.5,-5)[circle,fill=white!20, minimum size=.9cm] {$v_2$};
 \node[main node] (z13)at (2.5,-5) [circle,fill=white!20, minimum size=.9cm] {$y$};
 \node[main node] (z23)at (3.5,-5)[circle,fill=white!20, minimum size=.9cm] {$v_2$};
 
 \node[main node] (w12)at (-4,-6)[circle,fill=white!20, minimum size=.9cm] {$x$};
 \node[main node] (w22)at (-3.5,-6)[circle,fill=white!20, minimum size=.9cm] {$z_2$};
 \node[main node] (w13)at (-2.5,-6)[circle,fill=white!20, minimum size=.9cm] {$x$};
 \node[main node] (w14)at (-2,-6)[circle,fill=white!20, minimum size=.9cm] {$z_2$};

 \node[main node] (w15)at (-1,-6)[circle,fill=white!20, minimum size=.9cm] {$x$};
 \node[main node] (w25)at (-0.5,-6)[circle,fill=white!20, minimum size=.9cm] {$z_1$};
 
 \node[main node] (y12)at (0.5,-6)[circle,fill=white!20, minimum size=.9cm] {$x$};
 \node[main node] (y22)at (1,-6)[circle,fill=white!20, minimum size=.9cm] {$z_1$};

 \node[main node] (y23)at (2,-6)[circle,fill=white!20, minimum size=.9cm] {$z_1$};
 \node[main node] (y14)at (2.5,-6)[circle,fill=white!20, minimum size=.9cm] {$z_2$};

 \node[main node] (y15)at (3.5,-6)[circle,fill=white!20, minimum size=.9cm] {$z_1$};
 \node[main node] (y25)at (4,-6)[circle,fill=white!20, minimum size=.9cm] {$z_2$};
 
\draw (a1) edge (b1) edge (b2) edge (x)

(b1) edge (z11) edge (z21)
(b2) edge (z12) edge (z22)
(x) edge (z13) edge (z23)
 (z11) edge (w12) edge (w22)
 (z21) edge (w13) edge (w14)
 (z12) edge (w15) edge (w25)
 (z22) edge (y12) edge (y22)
 (z13) edge (y23) edge (y14)
 (z23) edge (y15) edge (y25);
 
 \end{tikzpicture}
\caption{A graph and its exploration tree (from $v_1$). Each vertex $w$ in the tree is labelled by $t(w)$.}
\label{extree}
\end{figure}

\begin{lem}
\label{leaves}
For $F$ a finite graph and $v \in V(F)$, let $T$ be the exploration tree from $v$ rooted at $v_0$. We have
$f_v(F) \le \frac{1}{2}L(v_0).$
\end{lem}

\begin{proof}
In the construction of $T$, for every induced cycle containing $v$ in $F$ we define two vertices of $T$. Again by construction, these two vertices are leaves of $T$. The result follows. 
\end{proof}

 Call a vertex $v \in T$ is \emph{good} if it has exactly three children: call it \emph{bad} otherwise. We now define a game on $T$, as we did previously in this section for a graph. We use the game to define vertices that are `(a)typical' for $T$. The following definition is the analogue in $T$ of Definition \ref{typG} for a graph.

\begin{defn}
Let $F$ be a finite graph and let $T$ be the exploration tree from $v$ in $F$. Let $w$ be a vertex in $V(F)\backslash N^4[v]$. We define the $w$-\emph{typical-game} on $T$ as follows. There are two players, Adversary and Builder. The game starts at vertex $u_0 := v_0$ ($v_0 = t(v))$ and the players choose a sequence of vertices $\{u_1,u_2, \ldots, u_k\}\subseteq V(T)$ under the following set of rules. At vertex $u_i$:
\begin{itemize}
\item If $t(u_i) \in N^4[w]$, then Adversary is the \emph{active player}, otherwise Builder is.
\item The active player chooses a child $u_{i+1}$ of $u_i$.
\item The game terminates when a vertex $u_j$ is chosen such that $u_j$ is a leaf.
\item Adversary wins if either for some $j$, we have that $t(u_j) \in N^4[w]$ and $u_j$ is bad, or if upon termination of the game at vertex $u_k$ there exists a vertex in $N^4[w]$ that is unseen by $\{u_0,\ldots, u_k\}$. Builder wins otherwise. 
\end{itemize} 
\end{defn}

A vertex $w \in V(F) \backslash N^4[v]$ is \emph{typical} for $T$ if there exists a winning strategy for Builder in the $w$-typical-game on $T$. A vertex is \emph{atypical} for $T$ otherwise. Observe that a vertex in $V(F) \backslash N^4[v]$ is atypical for $T(v)$ if and only if it is $v$-atypical in $F$. Also, note that a vertex $w$ being atypical for $T$ means that  Adversary has a strategy to ensure that, whatever strategy Builder chooses, either a bad vertex in $N^4[w]$ is chosen, or that there exists some vertex in $N^4[w]$ that remains unseen by $\{u_1,\ldots, u_k\}$ upon termination at vertex $u_k$. 

Now let $c>0$ and $F$ be any $n$-vertex graph with $f(F) \ge c \cdot 3^{n/3}$ and $\Delta(F) \le \Delta$, for some constant $\Delta$ ($G_{max}$ satisfies these conditions as $\Delta (G_{max})$ is bounded by Lemma \ref{deg}). Our next aim is to prove that, for any vertex $v \in V(F)$, only a bounded number of vertices are atypical for $T(v)$. Using this fact with Lemma \ref{structure} implies that the majority of the structure of $G_{max}$ is close to the structure of $H_n$. The remainder of the proof consists of `cleaning' $G_{max}$ to show that it is in fact isomorphic to $H_n$. 

We first outline how the proof will proceed before giving the details. We assume (in order to get a contradiction) that there is a large set $A \subseteq V(F)$ of vertices atypical for $T(v)$, such that for each $a, a' \in A$ we have $N^4[a]\cap N^4[a'] = \emptyset$ (for any set of atypical vertices there exists a subset of constant proportion with this property as $\Delta(F)$ is bounded). We will sequentially choose a path in $T$ of vertices $u_0, \ldots, u_k$ where $u_0 := v_0$ and $u_k$ is a leaf. 

For each $a \in A$, there exists a winning strategy $\tau_a$ for Adversary in the $a$-typical game on $T(v)$. This means that whatever vertices $u_i$ with $t(u_i) \notin \bigcup_{a \in A} N^4[a]$, are chosen in the path, for every $a \in A$ we are able to ensure that either:
\begin{itemize}
\item [(i)] we choose a bad vertex $u_i$ with $t(u_i) \in N^4[a]$, or
\item [(ii)] there is some vertex in $N^4[a]$ that remains unseen by $\{u_0,\ldots, u_k\}$.
\end{itemize}

We assume at the start that $L(v_0)$ is bounded below by $c\cdot 3^{n/3}$, for some constant $c$. As we move down the tree we keep track of the number of leaves that the branch we are in contains. If we are at a vertex $u_i$, such that $t(u_i)$ is not in $N^4[a]$ for any $a \in A$, we choose the branch that has the most leaves. When $t(u_i)$ is in $N^4[a]$ for some $a \in A$, we play the winning strategy $\tau_a$ to move towards the outcomes (i) or (ii), unless there is a sub-branch that contains a large proportion of the leaves in our current branch. These outcomes mean that the tree is `unbalanced' in some way, and the strategy that achieves these outcomes picks branches that contain more leaves than average. As it turns out, when we reach a leaf and the process ends, if $|A|$ was too large we find that the branch we are in ought to contain more than one leaf, a contradiction. 

\begin{lem}\label{genatypcount}
Fix $c>0$. Let $F$ be an $n$-vertex graph with $\Delta:= \Delta(F) > 1$ and let $\epsilon := 2^{-\Delta^{100}}$. Let $v \in V(F)$ and let $T = T(v)$ be the exploration tree from $v$ in $F$ with root $v_0$.  Let $A \subseteq V(F) \backslash N^4[v]$ be a set of atypical vertices for $T$ such that for all $a, a' \in A$, we have $N^4[a] \cap N^4[a'] = \emptyset$. If $L(v_0) \ge c\cdot 3^{n/3}$, then $|A| < M$, where $M$ is the smallest integer such that $c\cdot 3^{1/3}(1+\epsilon)^M > \Delta$. 
\end{lem}  

\begin{proof}

Suppose, in order to obtain a contradiction, that $|A| \ge M$. For each $a \in A$, Adversary has a winning strategy $\tau_a$ played on vertices of $N^4[a]$ in the $a$-typical game on $T$. As for all $a, a' \in A$, we have $N^4[a] \cap N^4[a'] = \emptyset$, these strategies are played on disjoint sets of vertices. 

We sequentially choose a path $v_0,u_1, \ldots, u_k$ of vertices through the tree where $u_k$ is a leaf. At each stage $i$, we choose a vertex $u_i$ and define $A_i$ (the subset of $A$ that we still care about tracking). We also define $C_i(a)$ for each $a \in A$. For each $a \in A$ define $C_1(a):= 1$.  Also define 
\begin{equation}\label{initqk}
C_i := \frac{c\cdot 3^{1/3}}{\Delta}\prod_{a \in A}C_i(a) \text{ and } q_i:= C_i3^{\frac{n-m_i-1}{3}},
\end{equation}
where $m_i$ is the number of vertices of $V(F)\backslash \{v\}$ seen by $V(P_{u_{i-1}})\backslash \{u_0\}$ (thus $n-m_i -1$ vertices of $V(F)\backslash \{v\}$ are unseen by $V(P_{u_{i-1}})\backslash \{u_0\}$). So $q_1 = \frac{c \cdot 3^{1/3}}{\Delta}3^{(n-1)/3}$. Throughout the process we maintain the property that $L(u_i) \ge q_i$ for each $i$. 

We now describe an algorithm that determines our choice of vertices. For $r\ge 1$, let $\epsilon_r = 2^{2(r-1)}\epsilon$. 

\noindent\rule[0.5ex]{\linewidth}{1pt}
Vertex Choice Algorithm

\noindent\rule[0.5ex]{\linewidth}{1pt}
We pick $u_1 \in N(v_0)$ such that $L(u_1)$ is maximised, and define $A_1 := A$.
Suppose the most recently chosen vertex is $u_i$ and that $m_i$ vertices of $V(F)\backslash \{v\}$ have been seen by $\{u_1,\ldots, u_{i-1}\}$.
If $u_i$ is not a leaf; we have two cases:

\begin{case2} $t(u_i) \in N^4[a]$ for some $a \in A_i$.
\end{case2} 
Suppose it is the $r$-th time we have chosen a vertex $y$ such that $t(y) \in N^4[a]$. We have two subcases:

\begin{subcase} $u_i$ is good.
\end{subcase}
In this case $u_i$ has exactly three children $y_1,y_2,y_3$. 
\begin{itemize}
\item [(i)] If there exists $j$ such that $L(y_j) \ge \frac{1}{3}(1 + \epsilon_r)C_i3^{(n-m_i-1)/3}$ then choose $u_{i+1} := y_j$. 

- Set $C_{i+1}(a):= (1 + \epsilon_r)C_i(a)$ and $C_{i+1}(y) := C_i(y)$, for all $y \in A\backslash\{a\}$.

- Set $A_{i+1} := A_i \backslash \{a\}$.

\item [(ii)] Else, every $y_j$ satisfies $L(y_j) > \frac{1}{3}(1 - 2\epsilon_r)C_i3^{(n-m_i-1)/3}$. In this case, choose $u_{i+1}$ according to strategy $\tau_a$.

- Set $C_{i+1}(a):= (1 - 2\epsilon_r)C_i(a)$ and $C_{i+1}(y) := C_i(y)$, for all $y \in A\backslash\{a\}$.

- Set $A_{i+1} := A_i$.

\end{itemize}
\begin{subcase} $u_i$ is bad.
\end{subcase}
In this case $u_i$ does not have exactly 3 children. Suppose $u_i$ has children $y_1, \ldots, y_k$ for some $k \not= 3$. Pick $j$ such that $L(y_j)$ is maximised and set $u_{i+1}:= y_j$.\vspace{0.2cm}

- Set $C_{i+1}(a) = (1 + \epsilon_r)C_i(a)$ and $C_{i+1}(y) := C_i(y)$, for all $y \in A\backslash\{a\}$.\vspace{0.2cm}

- Set $A_{i+1} := A_i \backslash \{a\}$.

\begin{case2}
$t(u_i) \not\in N^4[a]$ for any $a \in A_i$:
\end{case2}
Then $v$ has children $y_1,\ldots y_k$ for some $k \ge 1$. Pick $j$ such that $L(y_j)$ is maximised and set $u_{i+1} = y_j$.\vspace{0.2cm}

- Set $C_{i+1}(y) := C_i(y)$, for all $y \in A$. \vspace{0.2cm}

- Set $A_{i+1} := A_i$.\vspace{0.4cm}

The process terminates when $u_i$ is a leaf.

\noindent\rule[0.5ex]{\linewidth}{1pt}

We now analyse the consequences of choosing vertices in this manner.

\begin{claim}\label{qi}
For each vertex $u_i$ chosen during the Vertex Choice Algorithm, we have $L(u_i) \ge q_i$.
\end{claim}

\begin{proof}[Proof of Claim \ref{qi}]
 We argue by induction on $i$; the case $i= 1$ holds as we chose $u_1 \in N(v_0)$ to maximise $L(u_1)$. Suppose $L(u_i) \ge q_i = C_i3^{(n-m_i-1)/3}$. Now for the inductive step: we consider each case of the algorithm separately, and prove that the statement holds there. 

In Subcase 1(i) we have: 
$$L(u_{i+1}) \ge \frac{1}{3}(1+ \epsilon_r)C_i3^{\frac{n-m_i-1}{3}} = C_{i+1}3^{\frac{n-m_{i+1}-1}{3}} = q_{i+1}.$$

In Subcase 1(ii) we have: $$L(u_{i+1}) > \frac{1}{3}(1 - 2\epsilon_r)C_i3^{\frac{n-m_i-1}{3}} = C_{i+1}3^{\frac{n-m_{i+1}-1}{3}} = q_{i+1}.$$

In Subcase 2, recall that $u_i$ has neighbours $y_1, \ldots, y_k$ (for $k \not=3)$ and we pick $u_{i+1}$ to be the $y_j$ which maximises $L(y_j)$. Thus we have:
$$L(u_{i+1})\ge \frac{C_i}{k}3^{\frac{n-m_i-1}{3}} = C_i\frac{3^{k/3}}{k}3^{\frac{n-m_i-k-1}{3}}.$$
The value of $\frac{3^{k/3}}{k}$ is minimised for $k \not=3$ at $k = 2$. Thus,
$$L(u_{i+1}) \ge C_i\frac{3^{2/3}}{2}3^{\frac{n-m_i-k-1}{3}} \ge C_i(1 + \epsilon_r)3^{\frac{n-m_i-k-1}{3}} = C_{i+1}3^{\frac{n-m_{i+1}-1}{3}} = q_{i+1}.$$

In Case 2, recall that $u_i$ has neighbours $y_1, \ldots, y_k$ and we pick $u_{i+1}$ to be the $y_j$ which maximises $L(y_j)$. Thus we have:

$$L(u_{i+1}) \ge \frac{C_i}{k}3^{\frac{n-m_i-1}{3}} = C_i\frac{3^{k/3}}{k}3^{\frac{n-m_i-k-1}{3}} \ge C_i3^{\frac{n-m_i-k-1}{3}} = C_{i+1}3^{\frac{n-m_{i+1}-1}{3}} = q_{i+1},$$
where the last inequality is strict unless $k = 3$.
\end{proof}

We are now equipped to analyse what remains once the algorithm terminates at a leaf $u_k$. For each $a \in A$ at least one of the following outcomes occurs upon termination of the algorithm. 
\begin{itemize}
\item [(O1)] During the algorithm, at a vertex $a \in N^4[a]$, we either chose a branch with a large proportion of leaves via Case 1(i) or we chose a bad vertex via Subcase 2.
\item [(O2)] There is some vertex $w \in N^4[a]$ that is unseen by $V(P_{u_k})$ upon termination of the algorithm.
\end{itemize}

First observe that for all $s \le \Delta^5 + 1$,
\begin{equation}\label{epsil}
(1-2\epsilon_1)(1-2 \epsilon_2)\ldots(1-2\epsilon_{s})(1 + \epsilon_{s+1}) > (1 + \epsilon).
\end{equation}
Our choice of $\epsilon$ ensures that each factor on the left hand side is greater than zero.

Suppose $a \notin A_k$. Then there exists some $j$ such that $a \in A_j$ but $a \notin A_{j+1}$. Thus at the $j$th stage of the algorithm we had $t(u_j) \in N^4[a]$ and we either chose a branch with a large proportion of leaves via Subcase 1(i) or we chose a bad vertex via Subcase 2. Let $t:= |t(\{u_1,\ldots, u_j\}) \cap N^4[a]|$ be the number of vertices, $x \in T$ with $t(x) \in N^4[a]$, chosen up to the $j$th stage.
From the algorithm we see
\begin{equation}\label{ck}
C_k(a) = C_{j+1}(a) = (1-2\epsilon_1)(1-2 \epsilon_2)\ldots(1-2\epsilon_{t})(1 + \epsilon_{t+1}).
\end{equation}
As $|N^4[a]| \le \Delta^5$, we have $t \le \Delta^5$, and so by (\ref{epsil}) we have for $a \notin A_k$ 
\begin{equation}
\label{grr}
C_k(a) > (1 + \epsilon).
\end{equation}

Now suppose $a \in A_k$. Let $t:= |t(\{u_1,\ldots,u_{k-1}\})\cap N^4[a]|$. By following the algorithm we see that whenever we are at a vertex $u \in N^4[a]$, we do not pass through Subcase 1(i) or Subcase 2, as this would imply $a \notin A_k$. Thus $C_k(a) = \prod_{i=1}^t(1 - 2 \epsilon_{i})$. By choice of $\epsilon$, for all $s \le \Delta^5$ we have
$$3^{1/3} > 1 +2^{2s}\epsilon,$$ so by (\ref{epsil}) and the observation that $t \le \Delta^5$, 
\begin{equation}\label{ak}
3^{1/3}\cdot C_k(a) > (1 + \epsilon).
\end{equation} 

For each $a \in A_k$, the set $V(P_{u_k})$ does not see all of $N^4[a]$, as we either achieve outcome (O1) or (O2) for $a$, and if we achieved (O1), then $a$ would not be in $A_k$. So at termination we have $n - m_k -1 \ge |A_k|$ and so by the definition of $q_k$ (\ref{initqk}) we have:
\begin{equation}\label{qk}
q_k \ge C_k3^{\frac{|A_k|}{3}}.
\end{equation} By (\ref{grr}), we have: 
$$\prod_{a \in A}C_k(a) \ge (1 + \epsilon)^{|A\backslash A_k|}\prod_{a \in A_k}C_k(a),$$ and so by substituting for $C_k$ in (\ref{qk}) and applying (\ref{ak}), we have:
$$q_k \ge \frac{c\cdot 3^{1/3}}{\Delta} 3^{\frac{|A_k|}{3}} \prod_{a \in A} C_k(a) \ge \frac{c\cdot 3^{1/3}}{\Delta} (1 + \epsilon)^{M - |A_k|}3^{\frac{|A_k|}{3}}\prod_{a \in A_k} C_k(a) \ge \frac{c\cdot 3^{1/3}}{\Delta}(1+\epsilon)^{M} > 1.$$
This contradicts Claim \ref{qi}, as $u_k$ is a leaf and thus $L(u_k) = 1$. Thus $|A| < M$,  concluding the proof of Lemma \ref{genatypcount}.

\end{proof}

\begin{cor}\label{const} 
Fix $c>0$. Let $F$ be an $n$-vertex graph with $f(F) \ge 2c \cdot 3^{n/3}$ and $\Delta:= \Delta(F)>1$. Then for any $v \in V(F)$, there exists a constant $C = C(\Delta,c)$ such that at most $C$ vertices are atypical for $T(v)$, the exploration tree from $v$ in $F$.
\end{cor}

\begin{proof}
Let $A$ be the set of vertices that are atypical for $T$. Let $U$ be the largest subset of $A$ such that for all $a, a' \in U$, we have $N^4[a]\cap N^4[a'] = \emptyset$. As $|N^4[x]| \le \Delta^5$ for every $x \in V(F)$, we have 
\begin{equation}
\label{ua}
|U| \ge \frac{|A|}{\Delta^5}.
\end{equation}
We wish to apply Lemma \ref{genatypcount} to $F$, $T$ and $U$.

As $f(F) \ge 2c \cdot m(n)$, by Lemma \ref{vertex} $f_v(F) \ge \frac{c}{10}m(n)$ for all $v \in F$. Combining this with Lemma \ref{leaves} gives
$$L(v_0) \ge 2f_v(F) \ge \frac{c}{5} m(n)\ge \frac{c}{20}3^{n/3},$$
for some constant $c'$ and where the last inequality follows from (\ref{mn}). 

Applying Lemma \ref{genatypcount} shows that $|U| < M$, where $M$ is the smallest integer such that $\frac{c}{20}\cdot 3^{1/3}(1 + 2^{-\Delta^{100}})^M > \Delta$. Combining this with (\ref{ua}) gives the required result.

\end{proof}
Let $B= (B_1,\ldots, B_k)$ be a braid in $G_{max}$. If $|B_i| = 3$ for all $i$, we call $B$ a \emph{3-braid}. For a braid $B$ of length at least 4, we say that an induced cycle \emph{passes through} $B$ if it contains a vertex from every cluster of $B$. Call a braid \emph{maximal} if it is not contained in any longer braid. The following simple deduction will be used.

\begin{lem}\label{cn}
There exists a constant $C$ such that $G_{max}$ contains at most $C$ maximal 3-braids and a 3-braid $B$ such that $|V(B)| = \Omega(n)$. Moreover, for any 3-braid $B'$ on $rn$ vertices, at least $f(H_n)\left(1 - 3^{-rn/6}\right)$ induced cycles in $G_{max}$ pass through $B'$.
\end{lem}
\begin{proof}
Let $v \in V(G_{max})$. The only vertices which can be contained in more than one maximal 3-braid lie in end clusters. By Lemma \ref{structure}, every $v$-typical vertex is contained in a central cluster of exactly one maximal 3-braid. So any vertex in the end cluster of a maximal 3-braid is $v$-atypical. By Corollary \ref{const}, there exists a constant $c$ such that at most $c$ vertices are $v$-atypical for $G_{max}$. Each of these vertices is contained in at most $\Delta(G_{max}) \le 30$ maximal 3-braids. So $G_{max}$ contains (crudely) at most $30c$ maximal 3-braids, proving the first statement of the lemma. 
 
The union of the maximal 3-braids in $G_{max}$ contains all the typical vertices and so it contains at least $n/2$ vertices for large $n$. Therefore, when $n$ is sufficiently large, as there are at most $30c$ maximal 3-braids some 3-braid $B = (B_1,\ldots, B_k)$ contains $\Omega(n)$ vertices. 

For the final claim, observe that if an induced cycle does not pass through a 3-braid $\mathcal{B'}= (B'_1,\ldots,B'_k)$ on $rn$ vertices, then it is either a $C_4$ contained in $B$ (there are at most $O(n^4)$ of these), or it is contained in $V(G_{max})\backslash \bigcup_{i=3}^{k-2}B'_i$ (by Lemma \ref{upperv}, there are at most $[(1-r)n+12]\binom{30}{2}3^{[(1-r)n+9]/3}$ of these).
Therefore at most
$$[(1-r)n+12]\binom{30}{2}3^{[(1-r)n+9]/3} + O(n^4)$$
induced cycles of $G_{max}$ do not pass through $B$. So for $n_0$ sufficiently large, at least 
$$f(H_n)\left(1 - 3^{-rn/6}\right)$$ 
induced cycles pass through $B$.

\end{proof}

The next lemma shows that $G_{max}$ is a cyclic braid. It will remain to determine the cluster sizes and whether there are edges within the clusters  of $G_{max}$. 

\begin{lem}\label{cbraid}
$G_{max}$ is a cyclic braid.
\end{lem}

\begin{proof}
Let $B:= (B_1,\ldots, B_{Cn/3})$ be the longest 3-braid in $G_{max}$. Let $Q$ be the number of induced cycles in $G_{max}$ that pass through $B$. By Lemma \ref{cn}, 
\begin{equation}\label{q1}
Q \ge f(H_n)\left(1 - 3^{-Cn/6}\right).
\end{equation}
Now let $G':= G [V(G_{max}) \backslash \bigcup_{i=2}^{Cn/3 - 1} B_i]$. Let $x$ and $y$ be two new vertices and define $H$ to be the graph on vertex set $V(H):= V(G') \cup \{x,y\}$, and edge set 
$$E(H) := E(G') \cup \{xb: b \in B_1\} \cup \{yb: b \in B_{Cn/3}\}.$$
We have 
\begin{equation}\label{q2}
Q = 3^{(Cn-6)/3}p_2(H;x,y).
\end{equation}
Combining (\ref{q1}) and (\ref{q2}) gives
\begin{equation}\label{q3}
p_2(H;x,y) \ge 3^{-(Cn-6)/3}\cdot f(H_n)\left(1 - 3^{-Cn/6}\right).
\end{equation}

We now focus on the structure of $H$. Let us call a central cluster $C$ of a maximal 3-braid $B$ \emph{supercentral} if for any $x \in C$ and $y$ in an end cluster of $B$, $d(x,y) \ge 5$. Define a new graph $H'$ via the following process.
\begin{itemize}
\item Set $F_1:= H$.
\item Let $i$ be maximal such that we have defined $F_i$. Suppose there exists a vertex $v_i \in F_i$, contained in a supercentral cluster $C_i$ of a maximal 3-braid $M_i$, where $C_i$ is adjacent to clusters $C^i_1$ and $C^i_2$. Then define $F_{i+1}$ to be the graph obtained from $F_i$ by deleting $C_i$ and adding every edge $\{uw: u \in C^i_1, w \in C^i_2\}$. 
\item If there exists no such vertex $v_i$, define $H':= F_i$.
\end{itemize}
The process will terminate as $H$ has a finite number of vertices. Observe that $F_{i+1}$ is a braid if and only if $F_{i}$ is a braid.  In addition, when $H'$ is a braid this process can be reversed to find $H$. We now show that $H'$ is a braid. 

Any $v$-typical vertex in $F_i$ that does not get deleted during the process is $v$-typical in $F_{i+1}$. Hence any $v$-typical vertex in $G_{max}$ that does not get deleted is $v$-typical in $H'$. By Lemma \ref{cn}, there exists a constant $a$ such that $G_{max}$ contains $a$ maximal 3-braids. $O(a)$ vertices from each of these braids will remain in $H'$ when the process terminates. Any vertex not contained in a 3-braid in $G_{max}$ is $v$-atypical in $G_{max}$. By Corollary \ref{const} there exists a constant $b$ such that there are at most $b$ such vertices. As $H'$ contains all the atypical vertices of $G_{max}$ and at most $O(a)$ vertices from each 3-braid in $G_{max}$, there exists a constant $\beta$ such that $|V(H')| \le \beta$. 

At stage $i$ of the process, $F_{i+1}$ contains all induced cycles of $F_i$ that do not pass through $M_i$ and a third of the number of cycles in $F_i$ that do pass through $M_i$. Thus we have
\begin{equation}
p_2(F_{i+1};x,y) \ge 3^{-1} \cdot p_2(F_i;x,y),
\end{equation}

and so
\begin{equation}\label{q4}
p_2(H';x,y) \ge 3^{-(|V(H)| - |V(H')|)/3} \cdot p_2(H;x,y),
\end{equation}
Combining (\ref{q3}) and (\ref{q4}) and observing that $|V(H)| = (1-C)n + 8$ gives 
\begin{equation}\label{q6}
p_2(H';x,y) \ge 3^{(-n-2 + |V(H')|)/3} \cdot f(H_n)\left(1 - 3^{-Cn/6}\right).
\end{equation}
As $$3^{(-n-2 +|V(H')|)/3} f(H_n) = f_2(|V(H')|) + o(1),$$ when $n_0$ is sufficiently large we have
\begin{equation}\label{q7}
3^{(-n-2 + |V(H')|)/3} \cdot f(H_n)\left(1 - 3^{-Cn/6}\right) > f_2(|V(H')|) - 1.
\end{equation}
As $p_2(H';x,y)$ is an integer, by taking $n_0$ to be sufficiently large, (\ref{q6}) and (\ref{q7}) give
$$ p_2(H';x,y) \ge f_2(|V(H')|).$$

Therefore, by Theorem \ref{thm:path}, $H'$ is isomorphic to a graph in $\mathcal{F}_{|V(H')|}$. Thus $H'$ is a braid. By reversing the process applied above (adding back in the supercentral clusters) to recreate $H$ from $H'$, we see that $H$ is a graph in $\mathcal{F}_{|V(H)|}$, and hence $G_{max}$ is a cyclic braid.
\end{proof}

\begin{cor}\label{mainstru}
We have the following:
\begin{itemize}
\item when $n \equiv 0$ modulo 3, $G_{max}$ has exactly $n/3$ clusters of size 3;
\item when $n \equiv 1$ modulo 3, $G_{max}$ has either one cluster of size 4 and $(n-4)/3$ of size 3, or two of size two and $(n-4)/3$ of size 3;
\item when $n \equiv 2$ modulo 3, $G_{max}$ has exactly one cluster of size 2 and $(n-2)/3$ of size 3.
\end{itemize}
\end{cor}

We are now in a position to complete the proof of Theorem \ref{main}. The next lemma shows that the clusters in $G_{max}$ do not contain any edges, and thus we will prove the required result for $n \equiv$ 0 or 2 modulo 3. In the remaining case we will need a short argument to decide whether the graph contains two clusters of size two, or one of size four. In both cases, the arguments are essentially routine checks.

\begin{lem}\label{noedges}
When $n \equiv 0, 2$ modulo 3, no cluster of $G_{max}$ contains edges. 
\end{lem}

\begin{proof}
First observe, that if $e$ is an edge within a cluster, the only induced cycles containing $e$ can be triangles, either contained within the cluster, or containing exactly one vertex from a neighbouring cluster; or induced copies of $C_4$ within the cluster (in the case that the cluster contains 4 vertices). 

Let $B$ be a cluster adjacent to clusters $A$ and $C$. Suppose there exists an edge $e = uv$ where $u,v \in V(B)$. The edge $e$ is contained in at most $|A| + |C| + (|B| - 2)$ induced cycles within $G_{max}$. The graph $G' = G_{max}\backslash\{e\}$ will contain at least $|A||C|$ induced copies of $C_4$ (for any $x \in A$, $y \in C$, the set $\{x,y,u,v\}$ induces a $C_4$) that are not induced cycles in $G_{max}$. 

As $G_{max}$ does not contain both a cluster of size 2 and a cluster of size 4, we have
$$|A||C| > |A| + |C| + (|B| - 2),$$
unless $|B|=3$ and at least one of $|A|$ or $|C|$ is equal to 2. Except for this case, the number of induced cycles in $G' = G_{max}\backslash \{e\}$ is greater than the number of induced cycles in $G_{max}$, a contradiction.

Now suppose $|B|=3$ and suppose without loss of generality that $A = \{a_1,a_2\}$. First consider the case where $|C| = 3$. Suppose $B$ contains an edge $e=uv$. This edge is contained in at most 6 triangles in $G_{max}$. By the above argument, $A$ does not contain an edge. The graph $G' = G_{max}\backslash \{e\}$ will contain at least 7 induced copies of $C_4$ that are not induced cycles in $G_{max}$ (for any $x \in A$, $y \in C$, the sets $\{x,y,u,v\}$ and $\{a_1,a_2,u,v\}$ induce copies of $C_4$). Thus $f(G') > f(G_{max})$, a contradiction. 

The remaining case to consider is when $C = \{c_1,c_2\}$. If $B$ contains an edge $e$, this edge is contained in at most 5 triangles in $G_{max}$. The graph $G' = G_{max}\backslash \{e\}$ contains at least 6 induced copies of $C_4$ that are not induced cycles in $G_{max}$. Thus $f(G') > f(G_{max})$, a contradiction. So no cluster in $G_{max}$ contains an edge. 
\end{proof}

We have proved Theorem \ref{main} in the cases where $n \equiv 0$ or 2 modulo 3. It remains to prove the result in the case $n \equiv 1$ modulo 3. 

\begin{lem}\label{1mod3}
When $n \equiv 1$ modulo 3, $G_{max}$ is isomorphic to $H_n$.
\end{lem}

\begin{proof}
By Corollary \ref{mainstru} and Lemma \ref{noedges}, we know that $G_{max}$ is one of two empty cyclic braids. One possibility is that it is isomorphic to $H_n$. The other possibility is that $G_{max}$ is an empty cyclic braid $G_2$ with exactly two clusters of size 2, and the rest of size 3. An induced cycle in $H_n$ or $G_2$ either contains exactly one vertex from each cluster, or is an induced copy of $C_4$. In both $H_n$ and $G_2$, the number of induced cycles containing exactly one vertex from each cluster is $4 \cdot 3^{(n-4)/3}$. Thus the only difference in $f(H_n)$ and $f(G_2)$ comes from the number of induced copies of $C_4$.

There are two types of $C_4$. Type 1 contains vertices from exactly two clusters. Type 2 contains vertices from three clusters. The graph $H_n$ contains $3(n+5)$ induced type 1 cycles; $G_2$ contains at most $3n - 14$ of this form (fewer if the two clusters of size 2 are not adjacent). The graph $H_n$ contains $9(n+4)$ induced type 2 cycles; $G_2$ contains at most $9n -42$ of this form. Thus $H_n$ contains more induced cycles than $G_2$ and therefore $G_{max}$ is isomorphic to $H_n$. 
\end{proof}

\section{Proof of Theorem \ref{stab}}\label{sectIon: stab}
The proof of Theorem \ref{stab} follows the same lines as the proof of Theorem \ref{main}. Before proceeding with the details of the proof, we first give an outline of what is to come. Let $0 < \alpha < 1$ be any constant and let $F$ be an $n$-vertex graph containing at least $\alpha \cdot m(n)$ induced cycles. We will show that it is possible to delete a constant number of vertices from $F$ to give a graph $F'$ with maximum degree bounded by a constant. Applying Lemma \ref{genatypcount} to $F'$ then shows that the number of atypical vertices in $F'$ is bounded by a constant. The result will immediately follow. We cannot simply apply Lemma \ref{genatypcount} to $F$, as $F$ may contain vertices of arbitrarily large degree.

\begin{lem}\label{stabdeg}
Fix $0 < \alpha < 1$ and define $\Delta^* = \Delta^*(\alpha)$ to be the smallest integer such that $6\binom{\Delta^*}{2}3^{(1-\Delta^*)/3}< \alpha\cdot 4 \cdot 3^{-4/3}$. For $n$ sufficiently large, let $F$ be an $n$-vertex graph with $f(F) \ge \alpha \cdot m(n)$. Then there exists a constant $C=C(\alpha)$ such that deleting edges incident to $C$ vertices of $F$ gives a graph $H$ with $\Delta(H) \le \Delta^*$. Moreover, $f(H) \ge \frac{\alpha}{2} \cdot m(n)$.
\end{lem}

\begin{proof}
Suppose that $\Delta(F) > \Delta^*$ (else we are trivially done). We create a new graph $H$, with $\Delta(H) \le \Delta^*$, in the following manner. Define $F_1:= F$. Let $i$ be maximal such that $F_i$ has been defined. If there exists a vertex $v_i \in V(F_i)$ with $d(v_i) > \Delta^*$ then define $F_{i+1} := F_i \backslash \{v_i\}$. This process will terminate as $F$ has a finite number of vertices. Suppose the process terminates at a graph $F_j$. We have $\Delta(F_j) \le \Delta^*$. Define $H:= F_j$. 

To prove the first statement of the lemma it suffices to show that there exists some constant $C:= C(\alpha)$ such that $j = C$. To prove this, we will bound the size of $f(F) - f(H)$ and use this to show that when $j > C$, we have $f(F) < \alpha \cdot m(n)$, a contradiction. 

By Lemma \ref{upperv},
$$f_{v_i}(F_i) \le \binom{d_{F_i}(v_i)}{2}3^{(n- i- d_{F_i}(v_i) -1)/3} \le \binom{\Delta^*}{2}3^{(n- i- \Delta^* -1)/3},$$
where the second inequality follows as the function $\binom{x}{2}3^{-x/3}$ is decreasing for $x \ge 6$. 
So
\begin{align}
f(H) &= f(F) - \sum_{i=1}^{j-1}f_{v_i}(F_i) \nonumber \\
&\ge f(F) - \binom{\Delta^*}{2}3^{(n-\Delta^*+1)/3}\sum_{i=1}^{j-1}3^{- i/3} \nonumber \\
& \ge f(F) -3\binom{\Delta^*}{2}3^{(n-\Delta^*+1)/3} \label{alp}.
\end{align}
As $|V(H)| = n-j +1$, for some constant $c$ we have $f(H) \le c \cdot 3^{(n-j+1)/3}$  by (\ref{mn}). Combining this with (\ref{alp}) gives:
\begin{equation}\label{ff}
f(F) \le c \cdot 3^{(n-j+1)/3} + 3\binom{\Delta^*}{2}3^{(n-\Delta^*+1)/3}.
\end{equation}
There exists a constant $C$ such that, whenever $j > C$ and $n$ is sufficiently large:
\begin{equation}\label{fff}
c \cdot 3^{(n-j+1)/3} < \frac{1}{2}(\alpha\cdot 4 \cdot 3^{-4/3})3^{n/3}.
\end{equation}
Suppose that $j \ge C$ and let $n$ be sufficiently large. Using the definition of $\Delta^*$ and substituting (\ref{fff}) into (\ref{ff}), gives
$$f(F) < \alpha\cdot 4 \cdot 3^{-4/3}3^{n/3} < \alpha \cdot m(n),$$
 where the final inequality is implied by (\ref{leastB}). This contradicts the hypothesis that $f(F) \ge \alpha \cdot m(n)$. Therefore $j< C$, completing the proof of the first statement of the lemma. 

We now prove the second statement. By (\ref{alp}) we have
$$f(H) \ge f(F) - 3\binom{\Delta^*}{2}3^{(n-\Delta^*+1)/3}.$$
Given the definition of $\Delta^*$, we have $f(H) \ge \frac{\alpha}{2}m(n)$. 
\end{proof}

\begin{proof}[Proof of Theorem \ref{stab}]
Let $F$ be an $n$-vertex graph containing at least $\alpha \cdot m(n)$ induced cycles. By Lemma \ref{stabdeg}, there exist constants $c = c(\alpha)$ and $\Delta^* = \Delta^*(\alpha)$ such that deleting $c$ vertices from $F$ gives a graph $F'$ with $\Delta(F') \le \Delta^*$ and $f(F') \ge \frac{\alpha}{2}m(n)$. For any $v \in V(F')$, by Lemma \ref{vertex} we have $f_v(F) \ge \frac{\alpha}{20}$. Let $A$ be the set of $v$-atypical vertices in $F'$. By applying Lemma \ref{genatypcount} we deduce that $|A|<M$ (where $M = M(\alpha)$ is defined as in Lemma \ref{genatypcount}). By Lemma \ref{structure}, every $v$-typical vertex in $F'$ is contained in a central cluster of exactly one maximal 3-braid. We obtain $H$ from $F'$ by adding and deleting edges incident to vertices in $A$. The result follows with $C(\alpha) = c + M$. 
\end{proof}

\section{Induced odd or even cycles}\label{section:odd}
In this section we prove Theorem \ref{oddcycle}. The proofs of Theorem \ref{oddholes} and Theorem \ref{even} closely follow that of Theorem \ref{oddcycle}.

Given a graph $G$, define $f_o(G)$ to be the number of induced odd cycles contained within $G$. Similarly, for $v \in G$, define $f_o^v(G)$ to be the number of induced odd cycles in $G$ that contain $v$. We have:  
\begin{equation}\label{oddlow}
m_o(n) \ge f_o(G_n) = \Omega(3^{n/3}),
\end{equation}
where $G_n$ is defined as in Section 1.
The proof of Theorem \ref{oddcycle} follows from Theorem \ref{stab} and some arguments analogous to those used in Theorem \ref{main}. For the latter, we refer back to Sections 2 and 3 where necessary. The main difference is that, instead of applying Theorem \ref{thm:path}, we use Theorem \ref{oddpath}.
 
We fix a large constant $n_0$ and let $G$ be a graph on $n \ge n_0$ vertices that contains $m_o(n)$ induced odd cycles. In what follows we let $n_0$ be sufficiently large when required and we make no attempts to optimise the constants given in our argument.

\begin{proof}[Sketch proof of Theorem \ref{oddcycle}]
We first show that $\Delta(G)\le 35$ using analogous arguments to those in Theorem \ref{main}. 
 
Lemma \ref{vertex} holds (as (\ref{oddlow}) gives us the analogous bound to (\ref{leastB}) that we need). Thus every vertex is contained in at least $\frac{1}{20}m_o(n)$ induced odd cycles. Thus we have 
\begin{equation}\label{frel}
m_o(n+1) \ge \left(1 + \frac{1}{20}\right)m_o(n),
\end{equation}
as in (\ref{nvsn+1}).

We use the same argument as in Lemma \ref{deg}, replacing $m(n)$ with $m_o(n)$, to show that $\Delta(G) \le 35$ (we get a different value for $\Delta$ as we use the lower bound $m_o(G_n) \ge 3^{(n-8)/3}$ and this differs from the lower bound used for $m(n)$). 

By (\ref{oddlow}) we have $f_o(G_n) = \Omega(3^{n/3})$. Thus applying Theorem \ref{stab} shows that there exists a constant $c$ such that adding and deleting edges incident to $c$ vertices of $G$ gives a cyclic braid $H$ with the same cluster sizes as $H_n$. Using this and the knowledge that $\Delta(G)$ is bounded by a fixed constant, it is seen that $G$ contains a 3-braid $B$ of even length such that $|V(B)| = rn$ for some constant $r$.  

We now show that $G$ is a cyclic braid. We use essentially the same argument as in Lemma \ref{cbraid} with $f_o(G_n)$ in place of $f(H_n)$. However, when applying the process of deleting central clusters, we delete a pair of adjacent clusters at a time (to maintain the count of odd cycles). We again reach a graph $H'$ such that there exists a constant $\beta$ with $|V(H')|=\beta$. We make the analogous deductions from there to reach the bound 
$$f_2^o(n) \le p_2^o(H';x,y).$$  We then apply Theorem \ref{oddpath} to determine that $H' \in \mathcal{F}^o_{|V(H')|}$. Reversing the process of deleting central clusters to obtain $H$ from $H'$, we get that $H \in \mathcal{F}^o_{|V(H)|}$. Therefore $G$ is a cyclic braid, with clusters all of size 3 except: 
\begin{itemize}
\item [-]three clusters of size 2, when $n \equiv 0$ modulo 6;
\item [-]two clusters of size 2, when $n \equiv 1$ modulo 6;
\item [-]one cluster of size 2, when $n \equiv 2$ modulo 6;
\item [-] a single cluster of size 4, when $n \equiv 4$ modulo 6; and
\item [-]either two clusters of size 4 or four clusters of size 2, when $n \equiv 5$ modulo 6.
\end{itemize}
It remains to determine whether there are edges within the clusters, the relative positions of the clusters in the cyclic braid (in the cases where more than one cluster does not have size 3) and, in the case $n \equiv 5$ modulo 6, to determine the precise cluster sizes. Using arguments of a similar nature to those in Lemma \ref{noedges} and Lemma \ref{1mod3}, it can be checked that $G \cong G_n$ for every value of $n\ge n_0$.
\end{proof} 

Theorem \ref{oddcycle} determines which $n$-vertex graphs contain the maximum number of odd cycles. Following essentially the same argument we prove Theorem \ref{oddholes} and Theorem \ref{even}, which determine the family of $n$-vertex graphs that contain the maximum number of odd holes or even holes respectively. 

\begin{proof}[Sketch proof of Theorem \ref{oddholes} and Theorem \ref{even}]
We use the same argument as in the proof of Theorem \ref{oddcycle}. In the case of Theorem \ref{even}, the argument can be modified to consider even induced cycles rather than odd. The main difference is at the final stage, where we know $G$ is a cyclic braid and the possible cluster sizes in $G$. Changing the positions of clusters and edges within clusters can only affect the holes that do not contain a vertex from every cluster. Thus any hole that can be affected has size 3 or 4.  

For the odd hole case, the positions of the clusters and the existence of edges within clusters will not alter the number of odd holes (as any induced cycle with size 3 or 4 is not an odd hole). In the even case, a simple check shows the number of holes of size 4 (given the cluster sizes) is maximised when $G$ is isomorphic to $E_n$. 
\end{proof}

\section{Conclusion}\label{section:end}

For sufficiently large $n$, we have determined precisely which graphs on $n$ vertices contain the maximum 
number of induced cycles, the maximum number of odd or even induced cycles, and the maximum number of holes. 
However, there are a number of interesting related 
questions.

In our proofs above we make no attempts to optimise the value of $n_0$. We know that in some small cases, $H_n$ does not contain the maximum number of induced cycles \cite{bm,mr}. We 
believe Theorem \ref{main} ought to be true for $n_0 = 30$, but our proof gives 
a much larger number. There are several places where we could improve the bound, 
most notably by choosing a more careful strategy in Lemma \ref{genatypcount}.
However we omit the details as the bound would still be extremely large.

It is natural to consider induced cycles of some length that depends on 
$n$. Let $c(n,l)$ be the maximum number of length $l$ induced cycles that can 
be contained in a graph on $n$ vertices. Let $C(n,l)$ be the set of graphs 
containing $c(n,l)$ induced cycles of length $l$. 

\begin{ques}
For $l = l(n)$, what is $C(n,l)$?
\end{ques}

When $l$ is linear we believe the following should hold.

\begin{conj}
Fix $c \in (0,1)$. If $l(n) = \lceil cn \rceil$, then for sufficiently large $n$ 
the only graphs in $C(n,l)$ are cyclic braids of length $l$.
\end{conj}

Perhaps a similar result holds down to cycles of length $\Omega(\sqrt{n})$.

\begin{ques}
Suppose $l(n) > \sqrt{n}$. For sufficiently large $n$, are all graphs in 
$C(n,l)$ cyclic braids?
\end{ques}

Another related question is to ask about induced subgraphs which are 
subdivisions of some fixed graph $H$. 

\begin{ques}
Given a fixed finite graph $H$, what is the maximum number of induced 
subdivisions of $H$ that can be contained in a graph on $n$ vertices (and which 
graphs realise this maximum)?
\end{ques}

Theorem \ref{main} answers this question for $H=C_3$, but what happens for other graphs?
For instance, which graphs maximise the number of induced subdivisions of 
$K_{1,3}$?  For large $n$, are the extremal graphs always blowups of some subdivision of $H$?
The rooted version of the question is also interesting, where we consider
induced subdivisions of $H$ where the branch vertices are fixed
(for instance Theorem \ref{thm:path} is a result of this form for $H=K_2$).

Finally, we remark that the related problem of finding the graph on $n$ vertices that 
contains the most cycles (not necessarily induced) is trivial as $K_n$ is the extremal graph. 
However, the problem becomes interesting when we forbid certain subgraphs (see Arman, Gunderson and Tsaturian \cite{gund} and Morrison, Roberts and Scott  
\cite{cycles}).

\section{Acknowledgments}

We would like to thank Brendan McKay and Mike Robson for helpful discussions. We would also like to thank Andrew Treglown for bringing \cite{tuz} to our attention and the referees for their helpful comments.

\bibliography{induce}
	\bibliographystyle{amsplain}

%\bibliography{ohbagen}
%  \bibliographystyle{amsplain}
%  \nocite{*}

%\begin{thebibliography}{99}
%\frenchspacing

%{\small

%}

\end{document}